\documentclass[a4paper,11pt]{article}
\usepackage[utf8]{inputenc}
\usepackage[T1]{fontenc}
\usepackage{geometry}

\usepackage[italian,english]{babel}
\usepackage{amsfonts}
\usepackage{subcaption}
\usepackage{amssymb}
\usepackage{bbm}
\usepackage{amsmath}

\usepackage{comment}

\usepackage{amsthm}
\usepackage{mathtools}
\usepackage{graphicx}
\usepackage{dsfont}

\usepackage{float}

\theoremstyle{plain}
\newtheorem{thm}{Theorem}[section]
\newtheorem{lem}{Lemma}[section]
\newtheorem{prop}{Proposition}[section]

\theoremstyle{definition}
\newtheorem{defn}{Definition}[section]

\theoremstyle{remark}
\newtheorem{rem}{Remark}[section]

\usepackage{pgfplots}
\usepackage{tikz}

\usepackage{abstract}

\usepackage{todonotes}

\usepackage{xstring}
\usepackage{standalone}
\usepackage{amsfonts}
\usepackage{hyperref}

\newcommand{\tinyspace}{\hspace{0.1em}}
\newcommand{\nnl}{\notag\\}

\newcommand{\Reals}{\mathbb R}
\newcommand{\cadlag}{\ifmmode\mathcal D\else c\`adl\`ag \fi}
\newcommand{\Skorochod}{Skorokhod}
\newcommand{\defeq}{:=}
\newcommand{\id}{\mathrm{id}}

\newcommand{\Naturals}{\mathbb N}
\newcommand{\Integers}{\mathbb Z}
\newcommand{\dconv}{\xrightarrow{\:\: \text{d} \:\: }}
\newcommand{\asconv}{\xrightarrow{\:\: \text{a.s.} \:\: }}

\newcommand{\E}[1]{\mathbb E[#1]}
\newcommand{\inDJ}[1]{
    \IfEqCase{#1}{
        {1}{\qquad \text{in } (\cadlag,J_1)}
        {2}{\qquad \text{in } (\cadlag,J_2)}
    }[\PackageError{inDJ}{Undefined option to inDJ: #1}{}]%
}

\newcommand{\inDM}[1]{
    \IfEqCase{#1}{
        {1}{\qquad \text{in } (\cadlag,M_1)}
        {2}{\qquad \text{in } (\cadlag,M_2)}
    }[\PackageError{inDM}{Undefined option to inDM: #1}{}]%
}
\newcommand{\inDPJ}[1]{
    \IfEqCase{#1}{
        {1}{\qquad \text{in } (\cadlag^{+},J_1)}
        {2}{\qquad \text{in } (\cadlag^{+},J_2)}
    }[\PackageError{inDPJ}{Undefined option to inDPJ: #1}{}]%
}
\newcommand{\inDPM}[1]{
    \IfEqCase{#1}{
        {1}{\qquad \text{in } (\cadlag^{+},M_1)}
        {2}{\qquad \text{in } (\cadlag^{+},M_2)}
    }[\PackageError{inDPM}{Undefined option to inDPM: #1}{}]%
}

\newcommand{\RW}[1][n]{S_{#1}}
\newcommand{\avgRW}{\mu}
\newcommand{\fluidscaledRW}[1][n]{\bar{S}^{(#1)}}
\newcommand{\diffusionscaledRWnomean}[1][n]{\hat{S}^{(#1)}}
\newcommand{\diffusionscaledRW}[1][n]{\tilde{S}^{(#1)}}
\newcommand{\diffusionlimitRW}[1][\alpha]{W^{(#1)}}

\newcommand{\environment}[1][n]{\omega_{#1}}
\newcommand{\avgenvironment}{\nu}
\newcommand{\fluidscaledenvironment}[1][n]{\bar{\omega}^{(#1)}}
\newcommand{\diffusionscaledenvironmentnomean}[1][n]{\hat{\omega}^{(#1)}}
\newcommand{\diffusionscaledenvironment}[1][n]{\tilde\omega^{(#1)}}
\newcommand{\diffusionlimitenvironment}[1][\beta]{Z^{(#1)}}

\newcommand{\flight}[1][n]{Y_{#1}}
\newcommand{\fluidscaledflight}[1][n]{\bar{Y}^{(#1)}}
\newcommand{\diffusionscaledflight}[1][n]{\tilde{Y}^{(#1)}}
\newcommand{\diffusionscaledflightnomean}[1][n]{\hat{Y}^{(#1)}}

 \def \l {{\lambda}}

\newcommand{\R} {\mathbb{R}}

\newcommand{\Z} {\mathbb{Z}}
\newcommand{\N} {\mathbb{N}}

\newcommand{\ds} {\displaystyle}

\newcommand{\article}[3] {\textsc{{#1}}, {\itshape {#2}}, {{#3}}.}
\newcommand{\book}[3] {\textsc{{#1}}, {\itshape {#2}}, {{#3}}.}
\newcommand{\vol} {\textbf}
\newcommand{\eps} {\varepsilon}

\newcommand{\into} {\longrightarrow}

\title{Limit theorems for L\'evy flights \\[4pt]
on a 1D L\'evy random medium}

\author{
Samuele Stivanello
\thanks{Dipartimento di Matematica,
Universit\`a di Padova, Via Trieste 63,
35121 Padova, Italy. 
stivanel@math.unipd.it,  
bianchi@math.unipd.it, 
elenam@math.unipd.it}
\and
Gianmarco Bet
\thanks{Dipartimento di Matematica e Informatica “Ulisse Dini”, 
Universit\`a di Firenze, Viale Morgagni 65
50134 Firenze, Italy. gianmarco.bet@unifi.it}
\and
Alessandra Bianchi $^*$
\and
Marco Lenci
\thanks{Dipartimento di Matematica, Universit\`a di Bologna,
Piazza di Porta San Donato 5, 40126 Bologna, Italy. marco.lenci@unibo.it}
\thanks{Istituto Nazionale di Fisica Nucleare,
Sezione di Bologna, Via Irnerio 46,
40126 Bologna, Italy.}
\and 
Elena Magnanini $^*$
}

\date{Accepted for publication  in \emph{Electronic Journal of Probability} 
\\[4pt]
April 2021}

\begin{document}
\maketitle

\begin{abstract}
  We study a random walk on a point process given by an ordered array
  of points $( \omega_k, \, k \in \Z)$ on the real line. The distances
  $\omega_{k+1} - \omega_k$ are i.i.d.\ random variables in the domain of
  attraction of a $\beta$-stable law, with $\beta \in (0,1) \cup (1,2)$.
  The random walk has i.i.d.\ jumps such that the transition probabilities between
  $\omega_k$ and $\omega_\ell$ depend on $\ell-k$ and are given by the
  distribution of a $\Z$-valued random variable in the domain of attraction of an
  $\alpha$-stable law, with $\alpha \in (0,1) \cup (1,2)$. Since the defining
  variables, for both the random walk and the point process, are heavy-tailed,
  we speak of a \emph{L\'evy flight on a L\'evy random medium}. For all
  combinations of the parameters $\alpha$ and $\beta$, we prove the annealed
  functional limit theorem for the suitably rescaled process, relative to the optimal
  Skorokhod topology in each case. When the limit process is not c\`adl\`ag, we
  prove convergence of the finite-dimensional distributions. When the limit
  process is deterministic, we also prove a limit theorem for the fluctuations,
  again relative to the optimal Skorokhod topology.
  \\[6pt]
  \textbf{Mathematics Subject Classification (2020):} 60G50; 60G55; 60F17;
  82C41; 60G51
  \\[6pt]
  \textbf{Keywords:} Random walk on point process; L\'evy random medium;
  L\'evy flights; Stable distributions; Anomalous diffusion; Stable processes.
\end{abstract}

\section{Introduction}

The expression `L\'evy random medium' indicates a stochastic point process, in
some space, where the distances between nearby points have heavy-tailed
distributions. Processes of this kind have been receiving a surge of attention, of late,
both in the physical and mathematical literature; cf., respectively,
\cite{bfk, s, bcv, bdlv, zdk, vbb} and \cite{bcll, blp, ms, zamparo}. They model a variety
of situations that are of interest in the sciences. In particular, they are used
as supports for various kinds of random walks, in order to study phenomena of
anomalous transport and anomalous diffusion. An incomplete list of general or
recent references on this topic includes \cite{szf, krs, cgls, zdk, acor, roac, roap}.

The random medium that we consider in this paper is perhaps the most
natural choice for a L\'evy random medium in the real line: a sequence of
random points $\omega = ( \omega_k, \, k \in \Z)$, where $\omega_0=0$ and
the nearest-neighbor distances $\zeta_k = \omega_{k+1} - \omega_k$ are i.i.d.\ 
variables in the normal domain of attraction of a $\beta$-stable variable, with
$\beta \in (0,1) \cup (1,2)$. Here $\beta$ is the index of the
stable distribution, not the skewness parameter, which equals 1 because
$\zeta_k >0$.

A random walk $Y = (Y_n, \, n \in \N)$ takes place on $\omega$ according
to the following rule. Independently of $\omega$, there exists a random walk
$S = ( S_n, \, n \in \N)$
on $\Z$ with $S_0=0$ and i.i.d.\ increments in the normal domain of
attraction of an $\alpha$-stable variable, with $\alpha \in (0,1) \cup (1,2)$.
We define $Y_n := \omega_{S_n}$.
This means that the process $Y$
performs the same jumps as $S$, but on the marked points $\omega_k$
instead of $\Z$. For example, if a realization of $S$ is $(0, 3, -1, \ldots)$,
the process $Y$ starts at the origin of $\R$, then
jumps to the third marked point to the right of $0$, then to the first marked point
to the left of $0$, and so on. In other words, $S$ \emph{drives} the dynamics
of $Y$ on the L\'evy medium. For this reason we call it the \emph{underlying
random walk}.

Our process of interest is $Y$. We may describe it as a \emph{L\'evy
flight on a one-dimensional L\'evy random medium}. This phrase is
borrowed from the physical literature, where the term `L\'evy flight' usually
indicates a discrete-time random walk with long-tailed instantaneous
jumps. This is in contrast to a `L\'evy walk', which in general designates
a \emph{persistent}, continuous- or discrete-time, random walk with
long-tailed trajectory segments that are run at constant finite speed
\cite{zdk}. A L\'evy walk is often seen as an interpolation of a L\'evy flight. For
example, an important process from the standpoint of applications is
$X := ( X(t), \, t \in [0, +\infty))$, the unit-speed interpolation of $Y$. This means that,
for any realization of $Y$, a trajectory of $X$ starts at the origin and visits all
the points $Y_n$ in the given order, traveling between them with velocity $1$ or
$-1$, depending on $Y_{n+1}$ being to the right or to the left of $Y_n$,
respectively. The walk $X$ is a generalization of a system that first
appeared in the physical literature 20 years ago with the name
\emph{L\'evy-Lorentz gas} \cite{bfk} (more precisely, the L\'evy-Lorentz gas is
the case where the underlying random walk is simple). It was devised as a
one-dimensional toy model for the study of anomalous diffusion in porous
media \cite{l, bfk, bcv}. See \cite{bcll, blp, zamparo} for recent mathematical results.

There are several reasons to study our L\'evy flight on random medium. The
most self-serving, on the part of the present authors, is to build a basis to
investigate the properties of the associated L\'evy walk, as described above
(see the proofs in \cite{bcll, blp}). Also, $Y$ can be thought of as the limit of a
continuous-time random walk with resting times on the points $\omega_k$,
when the ratio between the speed of the walker and the typical resting time diverges.
This can be used to model a variety of situations where a given dynamics is very
fast compared to its ``decision times'', e.g., electronic signal on a network whose
nodes act as relatively slow processing stations; human mobility 
(assuming, as is often the case, that resting times are substantially longer than travel times);
etc. This particular model aside, there is no lack of general motivation for the
study of random walks on points processes, especially in light of the fact that the
topic is regrettably less developed than others in the field of random walks, with
the exception perhaps of random walks on percolation clusters \emph{et similia}.
For some interesting lines of research see, e.g., \cite{cf, cfp, k, br, z, r} and
references therein. A recent paper which we extend with the present work is
\cite{ms}.

In this paper we give \emph{annealed} limit theorems for $Y$ in all cases
$\alpha, \beta \in (0,1) \cup (1,2)$, identifying in each case both the scale 
$n^\gamma$ whereby
\begin{equation} \label{rescaledY}
  \left( \frac{Y_{\lfloor nt \rfloor}}{n^\gamma} \,,\, t \in [0,+\infty) \right)
\end{equation}
converges to a non-null limit, and the limit process.
In all cases we prove the optimal, or at least
morally optimal, functional limit theorem, meaning that we show
distributional convergence of the process with respect to 
(w.r.t.)\ the strongest Skorokhod
topology that applies there.
There are cases in which there can be no
convergence in the $J_1$ or $M_1$ topologies: in such cases we prove
convergence w.r.t.\ $J_2$. When the limit process is not \cadlag (or c\`agl\`ad)
we show convergence of the finite-dimensional distributions. Finally, in the cases
where the limit of (\ref{rescaledY}) is deterministic, we prove a functional limit
theorem for the corresponding fluctuations, again relative to the optimal
topology. 

\bigskip

The paper is organized as follows. In Sections \ref{subs-2-1} and 
\ref{subs-2-2} we describe the model and set the notation for the 
$J_1$ and $J_2$ Skorokhod topologies on spaces of 
c\`adl\`ag/c\`agl\`ad functions; 
in Section \ref{section-limit-process} we lay out basic 
limit theorems for the underlying random walk $S$ and the random 
medium $\omega$; in Section \ref{subs-2-4} we present our main results.
Finally, Section 3 contains all the proofs of the main theorems.

\paragraph*{Acknowledgments.} We thank Ward Whitt for discussing with us
the issue of the $J_2$-continuity of the addition map (cf.\ end of Section
\ref{subs-fluct}). This work was partly supported by the
joint UniBo-UniFi-UniPd project \emph{``Stochastic dynamics in disordered 
media and applications in the sciences''}. A.~Bianchi is partially supported by 
the PRIN Grant 20155PAWZB \emph{``Large Scale Random Structures''} 
(MIUR, Italy) and by the BIRD project 198239/19 \emph{``Stochastic processes 
and applications to disordered systems''} (UniPd). M.~Lenci is partially 
supported by the PRIN Grant 2017S35EHN \emph{``Regular and stochastic 
behaviour in dynamical systems''} (MIUR, Italy). E.~Magnanini thanks the 
Department of Mathematics of Universit\`a di
Bologna, to which she was affiliated when most of this work was done.

\section{Model and Results}

\subsection{Setup}
\label{subs-2-1}

As mentioned in the introduction, the L\'evy flight on random medium
that we consider is a random walk performed over the points of a certain
random point process. We proceed to define all the necessary constructions.

\paragraph{Random medium.}
Let $\zeta := \left( \zeta_i, \, i \in \Integers\right)$ be a sequence of i.i.d.\ 
positive random variables. We assume that the law of $\zeta_i$ belongs 
to the normal 
basin of attraction of a $\beta$-stable distribution, with $\beta\in (0,1) \cup (1,2)$. 
In the case $\beta\in(0,1)$, this means that, as $n \to +\infty$,
\begin{equation}
\label{stable}
\frac 1 {n^{1/\beta}} \sum_{i=1}^{n}\zeta_i \dconv Z^{(\beta)}_{1},
\end{equation}
where $Z^{(\beta)}_{1}$ is a stable variable of index $\beta$
and skewness parameter 1 (because $\zeta_i>0$). In the case 
$\beta\in(1,2)$ we have instead
\begin{equation}
\label{stable_center}
\frac 1 {n^{1/\beta}} \sum_{i=1}^{n}(\zeta_i - \avgenvironment) \dconv 
\tilde{Z}^{(\beta)}_{1},
\end{equation}
for a stable variable $\tilde{Z}^{(\beta)}_{1}$ of index 
$\beta$. In this case, necessarily, $\avgenvironment$ is 
the expectation of $\zeta_i$ and the skewness parameter is 0.

The random medium associated to $\left( \zeta_i, i \in \Integers \right)$ is 
defined to be:
\begin{equation} \label{omega-n}
\environment[0]=0, \qquad
\environment[k] = \begin{cases}
\sum_{i=1}^k \zeta_i & \text{ if } k > 0, \\
0 &  \text{ if } k = 0, \\
-\sum_{i=k}^{-1} \zeta_i & \text{ if } k < 0.
\end{cases}
\end{equation}
This determines a point process $\environment[] := \left(\environment[k],
\, k \in \Integers\right)$ on $\Reals$
that we call {\itshape L\'{e}vy random medium} to emphasize the fact 
that the distribution of $\zeta_i$ has a heavy tail. Each point 
$\environment[k]$ will be called a {\itshape target}.
In other words, the distances between neighboring targets are drawn 
according to independent random variables $\zeta_i$.

\paragraph{Underlying random walk.}
We consider a $\Integers$-valued random walk $\RW[] \defeq 
(\RW, \, n \in \Naturals)$, with $\RW[0] = 0$
and i.i.d.\ increments $\xi_i \defeq \RW[i] - \RW[i-1]$
that are independent of $\zeta$ (and thus of $\omega$).
In other words, $S$ is given by
\begin{equation} \label{S-n}
\RW[0] = 0, \qquad \RW = \sum_{i=1}^n \xi_i \quad\text{ for } n \in \Z^+.
\end{equation}
The law of $\xi_i$ belongs to the normal basin of attraction of 
an $\alpha$-stable distribution, with $\alpha \in(0,1)\cup (1,2)$. This means 
that convergences analogous to those given in \eqref{stable} and 
\eqref{stable_center} apply to the $\xi_i$, with limit random variables 
denoted by $W^{(\alpha)}_{1}$ and $\widetilde{W}^{(\alpha)}_{1}$, respectively. 
We will refer to $S$ as the {\itshape underlying} random walk.

\paragraph{Random walk on the random medium.} The {\itshape random walk on
the random medium} $Y := (Y_n,\, n\in\N)$ is defined to be:
\begin{equation} \label{Y-n}
\flight \defeq \environment[\RW], \quad n \in \Naturals.
\end{equation}
In other words, $Y$ performs the same jumps as $S$, but on the points of 
$\environment[]$; see Figure \ref{figura_intro} for a hands-on explanation.
In the following we will focus on the derivation of the asymptotic law of $Y$,
under suitable scaling, with respect to the probability measure 
$\mathbb{P}$ governing the entire system (medium and dynamics).
This is sometimes referred to as the \emph{annealed} or \emph{averaged} law of $Y$.
\begin{figure}[h]
\begin{center}
\includegraphics[width=13cm]{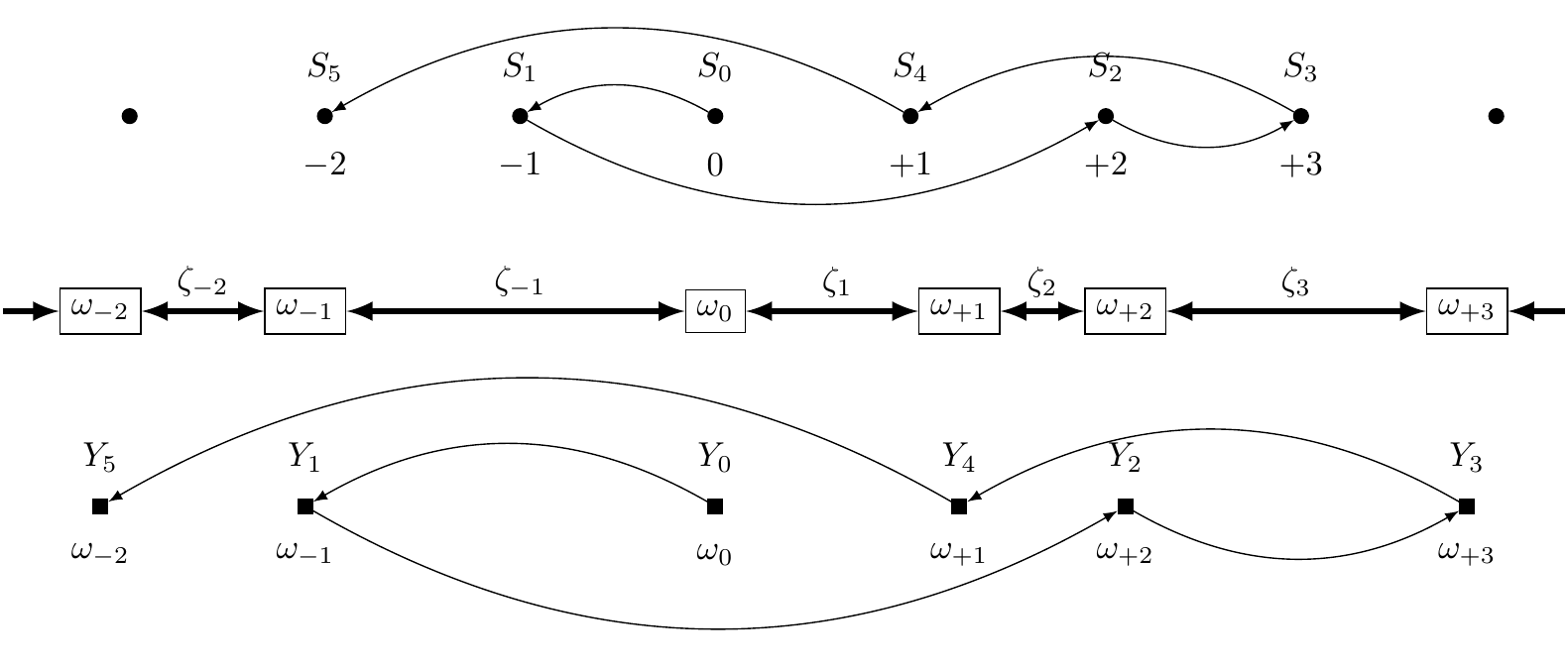}
\caption{Top: A realization of the underlying random walk $\RW[]$ on 
$\Integers$. Middle: A realization of the random medium $\omega$, 
with inter-distances given by $\zeta_i$.  Bottom: The corresponding 
process $Y$ jumps between the targets $\omega$ according to the 
walk $S$.}
\label{figura_intro}
\end{center}
\end{figure}

\bigskip

Before recalling certain basic facts about the processes 
$\omega$ and $\RW[]$, and stating our main results on the process 
$Y$, let us fix the notation for spaces of \cadlag 
functions endowed with certain Skorokhod topologies.

\subsection{C\`adl\`ag functions and Skorokhod topologies}
\label{subs-2-2}

Given $I$, an interval or a half-line contained in $\R^+ := [0, +\infty)$, 
we denote by 
$\cadlag(I) \equiv \cadlag(I ; \Reals)$ the space of all \cadlag functions 
$f: I \into \Reals$, where we recall that these are right-continuous
functions with left limits at all points of their domain. If $I$ is 
an interval or a half-line intersecting $(-\infty,0)$, or $I = \R$,
we consider a less customary function space: $\cadlag(I)$ is the 
space of all functions $f : I \into \R$ such that $s \mapsto f(s)$ is 
\cadlag for $s\ge0$ and $s \mapsto f(-s)$ is \cadlag for $s\ge0$ (in other 
words, the restriction of $f$ to $I \cap (-\infty,0]$ is c\`agl\`ad). Notice that
this implies that $f$ is continuous at 0. We also
use the abbreviations $\cadlag^+\equiv\cadlag(\Reals^+)$ and
$\cadlag \equiv \cadlag(\Reals)$. Lastly, we denote by 
$\cadlag_0$ and $\cadlag^+_0$ the subspaces of nondecreasing 
functions of $\cadlag$ and $\cadlag^+$, respectively.

In this section we introduce two notions of distance/topology that turn 
out to be crucial in the following. A complete treatment of these topologies 
can be found, e.g., in \cite[Sections 3.3. \& 11.5]{Whitt}.

\begin{defn} \label{def-j1-i}
Let $I$ be a bounded interval (which can be closed, open or half-open). For
$f,g \in \cadlag(I)$, denote
\begin{equation} \label{d-j1-i}
d_{J_1, I}(f,g) \defeq \inf_{\lambda: I \to I} \max \left\{ \sup_{t\in I} 
| f \circ \lambda(t) - g(t)| \:,\: \sup_{t\in I} | \lambda(t) - t| \right\},
\end{equation}
where the infimum is taken over all increasing homeomorphisms
$\lambda: I \into I$. This defines a distance on $\cadlag(I)$, which we
refer to as the $J_1$ or $J_1(I)$ distance.
\end{defn}

This metric induces a topology and a notion of limit in $\cadlag(I)$
which can be reformulated as follows: given $( f_n )_{n \in \Naturals}$ and
$f$ in $\cadlag(I)$, the sequence $f_n$ is said to converge to $f$ in the 
$J_1$ topology, and we write
\begin{equation} \label{first-time-conv-j1}
f_n \to f  \qquad \text{in } (\cadlag(I),J_1) ,
\end{equation}
as $n \to \infty$, if there exists a sequence of increasing homeomorphisms 
$\l_n: I \into I$ such that
\begin{align}
\label{eq:J_1_convergence_def_function}
\lim_{n\to\infty} \,&\sup_{t\in I} | f_n\circ\lambda_n(t) - f(t)| = 0, \\
\label{eq:J_1_convergence_def_time_change}
\lim_{n\to\infty} \,&\sup_{t\in I} | \lambda_n(t) - t| = 0.
\end{align}

\begin{defn} \label{def-top-j1-dp}
If $I \subset \R$ is a half-line, say $I= [a, +\infty)$, and 
$( f_n )_{n \in \Naturals}$, $f$ are functions of $\cadlag(I)$, we say that 
$f_n \to f$ in $(\cadlag(I),J_1)$, for $n\to\infty$, if, for all $T>a$ such that $f$ 
is continuous at $T$, 
\begin{equation} \label{marco14}
f_n \to f  \qquad \text{in } (\cadlag([a,T]),J_1).
\end{equation}
The analogous definition is given for $I = (a, +\infty)$ or $I = (-\infty, a]$, etc.
If $I=\R$, we say that $f_n \to f$ in $(\cadlag,J_1)$ if, for all $T>0$ such that 
$f$ is continuous at $T$ and $-T$, 
\begin{equation} \label{marco15}
f_n \to f  \qquad \text{in } (\cadlag([-T,T]),J_1).
\end{equation}
\end{defn}

The above definition defines a $J_1$ topology on $\cadlag(I)$, in all 
cases where $I$ is a half-line or the entire $\R$. It is easy to write a
metric that generates the $J_1(I)$ topology (see \cite[Section 2]{whitt1980some}).
\begin{rem} \label{rmk-cadlag-ified}
Although there are good reasons of convenience for using spaces
of functions that are \cadlag on $\R^+$ and c\`agl\`ad on $\R^-$ 
(see Section \ref{sec-proofs}),
some readers may find this choice odd and prefer to always work with 
\cadlag functions, even for domains $I$ intersecting $(-\infty,0)$. 
Clearly, to any $f \in \cadlag(I)$ as defined earlier there corresponds
a unique \cadlag version
\begin{equation} \label{cadlag-ified}
f_\mathrm{cadlag}(t) \defeq \lim_{s \to t^+} f(s).
\end{equation}
For any bounded $I$, it is easy to see that 
$d_{J_1,I}(f_\mathrm{cadlag}, g_\mathrm{cadlag}) =  d_{J_1,I}(f,g)$.
\end{rem}

\begin{rem} \label{rmk-0in0}
In this paper the only two cases in which we work with $I$ intersecting
$(-\infty,0)$ are $I= [-M,M]$ and $I=\R$. In both cases we only deal with 
functions $f$ such that $f(0)=0$. It is easy to see that,
under such additional condition, it is no loss of generality to require
that the homemorphism $\l$ fixes 0, i.e., $\l(0)=0$. This makes it
clear that, in such cases, $f_n \to f$ in $(\cadlag([-M,M]), J_1)$
if and only if both $f_n \to f$ and $f_n(-\: \cdot) \to f(-\: \cdot)$ in 
$(\cadlag([0,M]), J_1)$. Here $f(-\: \cdot)$ denotes the function
$t \mapsto f(-t)$.
\end{rem}

\medskip

If we think of $f_n$ as describing the spatial motion of some particle, the
function $\l: I \into I$ of \eqref{d-j1-i} is sometimes called the 
\emph{time change}. Requiring the time change to be a homeomorphism 
is occasionally too strong a condition. One has
a weaker topology if they only require that $\l$ be a (possibly
discontinuous) bijection:

\begin{defn} \label{def-j2}
If $I$ is a bounded interval and $f,g \in \cadlag(I)$, the $J_2$ or $J_2(I)$ 
distance $d_{J_2, I}(f,g)$ is defined as in the r.h.s.\ of \eqref{d-j1-i}, but
with the infimum taken over all bijections $\lambda: I \into I$. The notions
of $J_2$-convergence in all cases of $I$ are derived as seen earlier for
$J_1$.
\end{defn}

\begin{rem} \label{rmk-local-unif-conv}
It is a known and easy-to-prove fact that, if $I_0 \subseteq I$ is 
an interval at positive distance from the discontinuities of $f$, and 
$f_n \to f$ in $(\cadlag(I),J_i)$, for either $i=1$ or $i=2$, then 
$\sup_{t \in I_0} | f_n(t) - f(t)| \to 0$.
\end{rem}

\begin{rem} \label{rmk-open-closed}
The definition of limit in $(\cadlag([a, +\infty)),J_i)$ ($i=1,2$) amounts to
checking that $f_n \to f$ in $(\cadlag([a,T]),J_i)$, for all $T>a$ such 
that $f$ is continuous at $T$, see \eqref{marco14}. With the help of the 
previous remark, it is easy to see that this is tantamount to checking that 
$f_n \to f$ in $(\cadlag([a,T)),J_i)$, for all $T>a$ such that
$f$ is continuous at $T$. In the remainder (see for example Section 
\ref{subs3-3}) we will liberally switch between the two conditions, as is 
more convenient.
\end{rem}

\subsection{Limit processes for $\environment[]$ and 
\texorpdfstring{$\RW[]$}{the underlying walk and the environment}} 
\label{section-limit-process}

We now recall some elementary functional limit theorems for
suitable rescalings of the processes $\omega$ and $S$, cf.\ 
\eqref{omega-n} and \eqref{S-n}.

By definition, for all $k\in \Integers$, $\environment[k]$ is a sum of 
$|k|$ i.i.d.\ random variables $\zeta_i$ in the normal domain of 
attraction of a $\beta$-stable distribution. We first deal with the case 
$\beta \in (0,1)$.
For every $s \in \Reals$ we define
\begin{equation} \label{hat-omega}
\diffusionscaledenvironmentnomean(s) \defeq
\begin{dcases}
\frac{ \environment[\lfloor n s \rfloor]}{n^{1/\beta}} 
& \text{ if } s \geq 0, \\[4pt]
\frac{ \environment[\lceil n  s \rceil]}{n^{1/\beta}} & \text{ if } s < 0.
\end{dcases}
\end{equation}
Let $(Z^{(\beta)}_{\pm}(s), \, s \ge 0)$ be two i.i.d.\ \cadlag 
L\'evy $\beta$-stable processes such that $Z^{(\beta)}_{\pm}(0)=0$ 
and $Z^{(\beta)}_{\pm}(1)$ is distributed like
$Z^{(\beta)}_{1}$, as introduced in \eqref{stable} (these two
conditions uniquely determine the common distribution of the processes). 
Set
\begin{equation} \label{Z}
\diffusionlimitenvironment(s) \defeq
\begin{dcases}
Z^{(\beta)}_{+}(s) & \mbox{ if } s \geq 0, \\
-Z^{(\beta)}_{-}(-s) & \mbox{ if } s < 0.
\end{dcases}
\end{equation}
Then (see, e.g., \cite[Section 4.5.3]{Whitt}), as $n \to \infty$,

\begin{equation} \label{convergence-hat-omega} 
\diffusionscaledenvironmentnomean \dconv \diffusionlimitenvironment \inDJ{1}.
\end{equation}

When $\beta \in (1,2)$, the average distance $\avgenvironment := 
\E{\zeta_i}$ between successive targets is finite and positive by assumption. 
So, at first order, a Strong Law of Large Numbers holds. More in detail,
setting
\begin{equation} \label{bar-omega}
\fluidscaledenvironment (s) \defeq
\begin{dcases}
\frac{\environment[\lfloor n  s \rfloor]}{n} & \text{ if } s \geq 0, \\[4pt]
\frac{\environment[\lceil n  s \rceil]}{n} & \text{ if } s < 0,
\end{dcases}
\end{equation}
we have
\begin{equation} \label{convergence-bar-omega}
\fluidscaledenvironment \asconv \avgenvironment \tinyspace\id \inDJ{1},
\end{equation}
as $n\to\infty$. Here and in the rest of the paper $\id$ denotes the identity 
function, on whatever domain it is defined. Furthermore, a functional 
convergence similar to 
\eqref{convergence-hat-omega} holds for the fluctuations around this
Law of Large Numbers. More explicitly, for $s\in\R$, define
\begin{equation}\label{Omegatilde}
\diffusionscaledenvironment(s) \defeq
\begin{dcases}
\frac{\sum_{i=1}^{\lfloor ns \rfloor}(\zeta_{i} - \avgenvironment)}{n^{1/\beta}} 
& \text{ if } s \geq 0,\\[4pt]
\frac{-\sum_{i= \lceil ns \rceil}^{-1}(\zeta_{i} - \avgenvironment)}{n^{1/\beta}} 
& \text{ if } s < 0.
\end{dcases}
\end{equation}
Then, as $n\to\infty$,
\begin{equation} \label{convergence-tilde-omega}
\diffusionscaledenvironment \dconv \tilde{Z}^{(\beta)} \inDJ{1},
\end{equation}
where the process $\tilde{Z}^{(\beta)}$ is defined similarly to $Z^{(\beta)}$, 
cf.\ \eqref{Z},
but with $\tilde{Z}^{(\beta)}_{\pm}(1)$ distributed like $\tilde{Z}^{(\beta)}_{1}$, 
introduced in \eqref{stable_center}.

Analogous limit theorems hold for the continuous-time rescaled versions 
of the underlying random walk $\RW[]$. By definition,
$\RW$ is a sum of $n$ i.i.d.\ random variables $\xi_i$ in 
the normal domain of attraction of an $\alpha$-stable distribution.
We distinguish two regimes, depending on the values of $\alpha$
and $\avgRW:=\E{\xi_i}$ (when applicable). 

The first regime corresponds to the cases $\alpha \in (0,1)$, or 
$\alpha \in (1,2)$ with $\avgRW = 0$. In these situations, the drift
of the underlying random walk is either undefined or null. In either
case, it does not affect the convergence of the process
\begin{equation} \label{hat-S}
\diffusionscaledRWnomean (t) \defeq  \frac{ \RW[\lfloor n t \rfloor]}{n^{1/\alpha}} ,
\end{equation}
which we define for $t \ge 0$. In fact, let $\diffusionlimitRW$ denote a L\'evy 
$\alpha$-stable process with $\diffusionlimitRW(0)=0$ and 
$\diffusionlimitRW(1)$ distributed like $W^{(\alpha)}_{1}$ (the latter
variable has been defined after \eqref{S-n}). Then, as $n \to \infty$,
\begin{equation} \label{convergence-hat-S}
\diffusionscaledRWnomean \dconv \diffusionlimitRW \inDPJ{1}.
\end{equation}

When $\alpha \in (1,2)$ and $\mu \ne 0$, set, for $t \ge 0$,
\begin{equation} \label{bar-S}
\fluidscaledRW (t) \defeq \frac{ \RW[\lfloor n  t \rfloor]}n.
\end{equation}
By the functional version of the Strong Law of Large Numbers,
\begin{equation} \label{convergence-bar-S}
\fluidscaledRW \asconv \avgRW \tinyspace\id \inDPJ{1} ,
\end{equation}
as $n\to \infty$. As for the fluctuations, defining
\begin{equation}\label{Stilde}
\diffusionscaledRW(t)\defeq \frac{\sum_{i=1}^{\lfloor nt \rfloor}
(\xi_{i}-\avgRW)}{n^{1/\alpha}} ,
\end{equation}
we get
\begin{equation} \label{convergence-tildeS} 
\diffusionscaledRW \dconv \widetilde{W}^{(\alpha)} \inDPJ{1}.
\end{equation}
where $\widetilde{W}^{(\alpha)}$ is a L\'evy $\alpha$-stable process 
with $\widetilde{W}^{(\alpha)}(0) = 0$ and
$\widetilde{W}^{(\alpha)}(1)$ distributed like $\widetilde{W}^{(\alpha)}_{1}$
(again defined after \eqref{S-n}).

\subsection{Results}
\label{subs-2-4}

We now present our convergence results for the L\'{e}vy flight $\flight[]$
which, as we shall see, strongly depend on the values of $\alpha$ and $\beta$.
All theorems are stated using the
notation established in the previous section.

We first analyze the case $\beta \in (0,1)$, corresponding to an infinite
expected distance between the targets of the random medium 
\begin{thm} \label{th-1}
Let $\beta \in (0,1)$ and assume that either $\alpha \in (0,1)$ or 
$\alpha \in (1,2)$ with $\avgRW=0$. For $t \in \R^+$ define
\begin{equation} \label{Y-omegahat-Shat}
\diffusionscaledflightnomean(t)
\defeq
\diffusionscaledenvironmentnomean \circ \diffusionscaledRWnomean(t)
=
\frac{Y_{\lfloor n  t \rfloor}}{n^{1/\alpha \beta}},
\end{equation}
where $\diffusionscaledenvironmentnomean$ and 
$\diffusionscaledRWnomean$ have been introduced, respectively,
in \eqref{hat-omega} and \eqref{hat-S}.
Then the finite-dimensional distributions of $\diffusionscaledflightnomean$ 
converge to those of $\diffusionlimitenvironment \circ \diffusionlimitRW$, 
i.e., for any $m \in \Z^+$ and $t_1,\ldots,t_m \in \Reals^+$,
\begin{equation} \label{eq1-th-1}
\left( \diffusionscaledflightnomean (t_1) , \ldots , 
\diffusionscaledflightnomean (t_m) \right)
\dconv
\left( \diffusionlimitenvironment(\diffusionlimitRW (t_1)) , \ldots , \diffusionlimitenvironment(\diffusionlimitRW (t_m)) \right),
\end{equation}
as $n \rightarrow \infty$.
\end{thm}
Theorem \ref{th-1} is rather weak, in that it only proves convergence of 
the finite-dimensional distributions of the process 
$\diffusionscaledflightnomean$ defined in \eqref{Y-omegahat-Shat}.
Observe, however, that the limit process 
$Z^{(\beta)}\circ\,W^{(\alpha)}$ has trajectories that are not \cadlag with 
positive probability (see for example the explanation around (2.9)
of \cite{blp}). Therefore, a functional limit theorem 
w.r.t.\ a Skorokhod topology is not the natural result to expect. On the other
hand, when $\alpha\in(1,2)$ and $\mu\neq 0$, the assertion can be 
strengthened as follows.
\begin{thm} \label{th-conv-j2}
Let $\beta\in(0,1)$ and $\alpha\in(1,2)$ with $\mu \neq 0$. For $t \in \R^+$ 
define
\begin{equation} \label{Y-omegahat-Sbar}
\diffusionscaledflightnomean(t)
\defeq
\diffusionscaledenvironmentnomean \circ \fluidscaledRW(t)
=
\frac{\flight[\lfloor n  t \rfloor]}{n^{1/ \beta}} ,
\end{equation}
cf.\ \eqref{hat-omega} and \eqref{bar-S}. Then, as $n\to\infty$,
\begin{equation} \label{eq-conv-j2}
\diffusionscaledflightnomean \dconv \mathrm{sgn}(\avgRW) \,
|\avgRW|^{1/\beta} \, \diffusionlimitenvironment_+ \inDPJ{2}.
\end{equation}
\end{thm}
\begin{rem} 
Since $\diffusionlimitenvironment_+ \stackrel{\mathrm{d}}{=} 
\diffusionlimitenvironment_-$, one could put either process in the r.h.s.\ of
\eqref{eq-conv-j2}, irrespectively of the sign of $\avgRW$.
\end{rem}


\begin{rem} \label{first-rem-j2}
The convergence \eqref{eq-conv-j2} fails in the topology $J_1$, or 
even $M_1$ \cite[Section 3.3]{Whitt}. The topology $J_2$ 
is thus the strongest among the classical \Skorochod\ topologies with 
respect to which the convergence holds. To justify the claim, observe
that, in general, $\fluidscaledRW$ is a wildly oscillating function around 
$\avgRW \tinyspace \id$, and $\diffusionlimitenvironment$ is almost surely 
discontinuous. More in detail, assume that $\avgRW > 0$ and let 
$s \in \Reals$ be a discontinuity 
point of $\diffusionlimitenvironment_+$ with a jump, say, of order 1 in $n$.
Since, for $n \to \infty$, $\diffusionscaledenvironmentnomean$ is very 
close to $\diffusionlimitenvironment_+$ in $J_1$, there exists a discontinuity
point $s_n$ of $\diffusionscaledenvironmentnomean$, very close to $s$, 
with a jump of order 1. Now, if we exclude the case where the underlying 
random walk $S$ is deterministic, $\fluidscaledRW (t)$ is a non-monotonic 
function of $t \in I$, for every interval $I \subset \R^+$ and $n$ large
enough, depending on $I$ (this is an elementary Brownian-bridge
result). So one can find a small interval $I$ such that, as $t$ runs through $I$, 
$\fluidscaledRW(t)$ oscillates many times around  $s_n$. Therefore 
$\diffusionscaledenvironmentnomean \circ \fluidscaledRW (t)$ 
has many back-and-forth jumps of order 1. This prevents convergence
both in $J_1$ and in $M_1$, cf.\ \cite[Figure 11.2]{Whitt}.
What allows for $J_2$-convergence is that the fluctuations of 
$\fluidscaledRW$ around $\avgRW \tinyspace \id$ vanish, as $n \to \infty$. 
This means that the oscillations of $\fluidscaledRW(t)$ around $s_n$,
and therefore the large oscillations of 
$\diffusionscaledenvironmentnomean \circ \fluidscaledRW (t)$,
occur only in a vanishing interval $I_n \subset I$. Therefore 
one can find a non-continuous change of the coordinate $t$, 
say $\rho_n : [0,T) \into [0,T)$, which is globally 
close to the identity and ``reorders'' the points in $I_n$ 
in the sense that $\diffusionscaledenvironmentnomean \circ 
\fluidscaledRW \circ \rho_n$ only has one jump of order 1. The problem thus 
reduces to the much easier problem of showing the $J_1$-convergence
of the latter process. See the proof of Theorem \ref{th-conv-j2} for the 
rigorous arguments. 
Lastly, we observe that all the results presented in
this paper involving the $J_2$ topology could in fact be stated for a 
stronger \Skorochod-type topology. We refer the interested reader to 
Remark \ref{rmk-j32} of the Appendix.
\end{rem}

Next we consider the case $\beta \in (1,2)$, where the inter-distances 
of the random medium have finite mean.

\begin{thm} \label{th-2}
Let $\beta \in (1,2)$ and recall the notation \eqref{bar-omega}, \eqref{hat-S} 
and \eqref{bar-S}.
\begin{enumerate}
\item Assume $\alpha \in (0,1)$, or $\alpha \in (1,2)$ with $\mu=0$. For
$t \in \R^+$ set
\begin{equation} \label{Y-omegabar-Shat}
\diffusionscaledflightnomean(t) \defeq \fluidscaledenvironment \circ \diffusionscaledRWnomean(t) =
\frac{\flight[\lfloor n  t \rfloor]}{n^{1/\alpha}}.
\end{equation}
Then, as $n\to\infty$,
\begin{equation} \label{eq-th-2-a}
\diffusionscaledflightnomean   \dconv  \avgenvironment \tinyspace 
\diffusionlimitRW \inDPJ{1}.
\end{equation}

\item Assume $\alpha \in (1,2)$ and $\avgRW \ne 0$. Setting
\begin{equation} \label{Y-bar}
\fluidscaledflight(t)
\defeq
\fluidscaledenvironment \circ \fluidscaledRW(t)
=
\frac{\flight[\lfloor n  t \rfloor]}{n} 
\end{equation}
one has
\begin{equation} \label{eq-th-2-b}
\fluidscaledflight \dconv \avgenvironment \avgRW \tinyspace \id \inDPJ{1}.
\end{equation}
\end{enumerate}
\end{thm}

As stated in point 2 above, when $\alpha \in (1,2)$ and $\avgRW \neq 0$, 
the sequence of processes $\fluidscaledflight$ converges to a multiple 
of the identity function.
The next theorem gives the explicit asymptotics of the fluctuations of 
$\fluidscaledflight$ around its deterministic limit.
\begin{thm} \label{th-3}
Let $\alpha, \beta \in (1,2)$ with $\avgRW \ne 0$, and let 
$\fluidscaledflight$ be the process defined in \eqref{Y-bar}.

\begin{enumerate}

\item If $\alpha<\beta$ define
\begin{align}\label{Ytilde1}
\diffusionscaledflight(t) \defeq\frac{ n(\fluidscaledflight(t) - 
\avgenvironment \avgRW \tinyspace t)} {n^{1/\alpha}}.
\end{align}
Then, when $n \to \infty$, 
%
\begin{equation} \label{eq:fCLT_alpha<beta}
\diffusionscaledflight
\dconv
\avgenvironment  \widetilde{W}^{(\alpha)}  \inDPJ{1},
\end{equation}
where $\widetilde{W}^{(\alpha)}$ has been defined after 
\eqref{convergence-tildeS}. 

\item If $\alpha > \beta$ define
\begin{align}\label{Ytilde2}
\diffusionscaledflight(t)\defeq\frac{n(\fluidscaledflight(t) - 
\avgenvironment \avgRW \tinyspace t)}{n^{1/\beta}}.
\end{align}
Then, when $n \to \infty$, 
%
\begin{equation} \label{eq:fCLT_alpha>beta}
\diffusionscaledflight
\dconv
\mathrm{sgn}(\avgRW) \, |\avgRW|^{1/\beta} \, \tilde{Z}^{(\beta)}_{+} \inDPJ{2},
\end{equation}
where $\tilde{Z}^{(\beta)}_{+}$ has been defined after 
\eqref{convergence-tilde-omega}.
\item If $\alpha = \beta$ define
\begin{align}\label{Ytilde3}
\diffusionscaledflight(t)\defeq\frac{ n(\fluidscaledflight(t)  - \avgenvironment \avgRW \tinyspace t)} {n^{1/\alpha}}.
\end{align}
Let $\tilde{Z}^{(\alpha)}_{+}$ and $\widetilde{W}^{(\alpha)}$ be two 
independent $\alpha$-stable processes, as previously defined. As
$n\to\infty$,
\begin{equation}\label{eq:fCLT_alpha_equals_beta}
\diffusionscaledflight
\dconv
\mathrm{sgn}(\avgRW) \, |\avgRW|^{1/\beta} \, \tilde{Z}^{(\alpha)}_{+} + 
\avgenvironment \widetilde{W}^{(\alpha)} \inDPJ{2}.
\end{equation}
\end{enumerate}
\end{thm}

\begin{rem}
The same considerations as in Remark \ref{first-rem-j2} apply to the
optimality of the $J_2$ topology in the limits \eqref{eq:fCLT_alpha>beta} 
and \eqref {eq:fCLT_alpha_equals_beta}.
\end{rem}

\section{Proofs} \label{sec-proofs}

\subsection{Proof of Theorem \ref{th-1}: Convergence of 
finite-dimensional distributions }

We establish the assertion by extending the proof of \cite[Theorem 2.2]{blp}.
We first prove the following:
\begin{lem} \label{lemma31}
Let $\diffusionscaledenvironmentnomean$ and $\diffusionscaledRWnomean$ 
be the processes defined in \eqref{hat-omega} and \eqref{hat-S}, respectively. Then, when $n\to\infty$,
\begin{equation}\label{eq:joint-fCLT}
( \diffusionscaledenvironmentnomean , 
\diffusionscaledRWnomean)\dconv
(\diffusionlimitenvironment , 
\diffusionlimitRW)\qquad \mbox{in } 
(\cadlag \times \cadlag^+,\, J_1 \otimes J_1)\,,
\end{equation}
where $J_1 \otimes J_1$ denotes the product topology on the product space 
$\cadlag \times \cadlag^+$.
\end{lem}
%

%
\begin{proof}
From \eqref{convergence-hat-omega} and \eqref{convergence-hat-S}
we have that $\diffusionscaledenvironmentnomean \dconv 
\diffusionlimitenvironment$ in $(\cadlag,J_1)$ and
$\diffusionscaledRWnomean \dconv \diffusionlimitRW$ in $(\cadlag^{+},J_1)$. 
Since $\diffusionscaledenvironmentnomean$ and
$\diffusionscaledRWnomean$ are independent, the result follows from 
\cite[Theorem 11.4.4]{Whitt}.
\end{proof}

By virtue of the Skorokhod Representation Theorem, we may assume that the 
convergence in the statement of Lemma \ref{lemma31} holds almost everywhere. 
If this is not the case, there exists a probability space where it does, and since 
the specifics of the probability space are irrelevant for the next discussion, we 
avoid here to change the notation for the processes in the new space. 
Notice also that since $\diffusionlimitenvironment$ is a $\beta$-stable process, 
it is almost surely continuous at $s$, for any $s \in \Reals$, and similarly 
$\diffusionlimitRW$ is almost surely continuous at $t$, for any $t \in \Reals^+$.
In particular, by the independence of the two processes,
the event that $\diffusionlimitRW$ is continuous 
at $t$ and $\diffusionlimitenvironment$ is continuous at 
$\diffusionlimitRW(t)$ has probability 1, for any $t \in \Reals^+$.
Therefore the hypotheses of the next lemma hold almost surely.

\begin{lem} \label{lemma33}
Fix $t >0$ and consider a realization $(\omega, S)$ of the random
medium and of the underlying random walk 
such that $\diffusionlimitRW$ is continuous in t and 
$\diffusionlimitenvironment$ is continuous at $\diffusionlimitRW(t)$.
Then we have
\begin{equation} \label{eq41}
\lim_{n \rightarrow \infty} \diffusionscaledenvironmentnomean 
( \diffusionscaledRWnomean (t)) = 
\diffusionlimitenvironment(\diffusionlimitRW(t)) .
\end{equation}
\end{lem}

\begin{proof}
Let $\varepsilon \in (0,1)$ and $\eta \in (0,\varepsilon)$ be such that
\begin{equation} \label{cont-Z}
\sup_{s : \vert s - \diffusionlimitRW(t) \vert < 2 \eta} \vert 
\diffusionlimitenvironment(\diffusionlimitRW(t)) - \diffusionlimitenvironment(s) 
\vert < \frac{\varepsilon}{2}.
\end{equation}
Also choose $\varsigma \in (0, \eta)$ so that
\begin{equation} \label{cont-W}
\sup_{u : \vert u - t \vert < \varsigma} \vert \diffusionlimitRW(t) - 
\diffusionlimitRW(u) \vert < \frac{\eta}2.
\end{equation}

Let $n$ be large enough so that  $d_{J_1,[0,t+1]}(\diffusionscaledRWnomean ,\diffusionlimitRW) < 
\varsigma / 2$, see \eqref{d-j1-i}. In other words, there exists an increasing 
homeomorphism $\varphi_n$ of $[0, t + 1]$ such that, for all $u \in [0,t+1]$,
\begin{align} \label{J1-SW-1}
\vert u - \varphi_n(u) \vert < \frac{\varsigma}{2}  ,\\[4pt]
\vert \diffusionscaledRWnomean (u) - \diffusionlimitRW (\varphi_n(u)) 
\vert < \frac{\varsigma}{2}.\label{J1-SW-2}
\end{align}
Hence, using \eqref{J1-SW-2} and \eqref{cont-W} we get
\begin{equation} \label{inner-conv}
\begin{split}
&\vert \diffusionscaledRWnomean (t) - \diffusionlimitRW(t) \vert \\
&\qquad \le \vert \diffusionscaledRWnomean (t) - \diffusionlimitRW
(\varphi_n(t)) \vert + \vert \diffusionlimitRW(\varphi_n(t)) - \diffusionlimitRW(t) 
\vert \\
&\qquad < \frac{\varsigma}{2} + \frac{\eta}{2} < \eta,
\end{split}
\end{equation}
%
%
since $ \vert \varphi_n(t) - t \vert < \varsigma/2$. 
Assume moreover that $n$ is large enough so that
\begin{align}
d_{J_1, [0, \vert \diffusionlimitRW(t)\vert + 1]}
\left(\diffusionscaledenvironmentnomean,Z^{(\beta)}\right) < 
\frac{\eta}{2} , \\[4pt]
\label{marco16}
d_{J_1, [0, \vert \diffusionlimitRW(t)\vert + 1]} \left(
\diffusionscaledenvironmentnomean(- \: \cdot ),Z^{(\beta)}(- \: \cdot)\right) < 
\frac{\eta}{2},
\end{align}
where the notation in the l.h.s.\ of \eqref{marco16} was introduced in 
Remark \ref{rmk-0in0}. Then there exists an increasing 
homeomorphism $\psi_n$ of $[- \vert\diffusionlimitRW(t)\vert - 1,
\vert\diffusionlimitRW(t)\vert + 1]$,  with $\psi_n(0) = 0$, such that, for all 
$s \in [- \vert\diffusionlimitRW(t)\vert - 1, \vert\diffusionlimitRW(t)\vert + 1]$,
\begin{align} \label{J1-omegaZ-1}
\vert s - \psi_n(s) \vert < \frac{\eta}{2} , \\[4pt]
\vert \diffusionscaledenvironmentnomean(s) - Z^{(\beta)}(\psi_n(s)) \vert  < \frac{\eta}{2}.\label{J1-omegaZ-2}
\end{align}
Note also that \eqref{inner-conv} ensures that $\diffusionscaledRWnomean(t) \in [- \vert\diffusionlimitRW(t)\vert - 1,\vert\diffusionlimitRW(t)\vert + 1]$, so that by \eqref{J1-omegaZ-1} and \eqref{inner-conv},
\begin{align} \label{condition-cont-Z}
\vert \psi_n(\diffusionscaledRWnomean(t)) - \diffusionlimitRW(t) \vert
&\leq
\vert \psi_n(\diffusionscaledRWnomean(t)) - \diffusionscaledRWnomean(t) \vert + \vert \diffusionscaledRWnomean(t) - \diffusionlimitRW(t) \vert\nnl
 &< \frac{\eta}2 + \eta < 2 \eta\,,
\end{align}
and from \eqref{J1-omegaZ-2} we get
\begin{equation} \label{medium-conv}
\vert \diffusionscaledenvironmentnomean ( \diffusionscaledRWnomean (t)) - \diffusionlimitenvironment(\psi_n(\diffusionscaledRWnomean (t))) \vert < \eta/2.
\end{equation}
Finally, using \eqref{medium-conv}, \eqref{cont-Z} and 
\eqref{condition-cont-Z}, we obtain:
\begin{align}
\vert &\diffusionscaledenvironmentnomean ( \diffusionscaledRWnomean (t)) - \diffusionlimitenvironment(\diffusionlimitRW(t)) \vert \nnl
&\leq
\vert \diffusionscaledenvironmentnomean ( \diffusionscaledRWnomean (t)) - \diffusionlimitenvironment(\psi_n(\diffusionscaledRWnomean(t))) \vert + \vert \diffusionlimitenvironment(\psi_n(\diffusionscaledRWnomean(t))) - \diffusionlimitenvironment(\diffusionlimitRW(t)) \vert \nnl
&\leq \frac{\eta}{2} + \frac{\varepsilon}2 < \varepsilon.
\end{align}
This shows \eqref{eq41}. 
\end{proof}

\begin{proof}[Proof of Theorem \ref{th-1}]
Let $m \in \mathbb{N}^+$ and $t_1,\ldots,t_m \in [0,+\infty)$. With probability 
one $\diffusionlimitRW$ is continuous at $t_1,\ldots, t_m$ and  
$\diffusionlimitenvironment$ is continuous at $\diffusionlimitRW(t_1),\ldots, 
\diffusionlimitRW(t_m)$. When restricting to such realizations, using
Lemma \ref{lemma33} with $t = t_i$, we have that 
$n^{-1/ \alpha \beta} \, Y_{\lfloor n \, t_i \rfloor} =
\diffusionscaledenvironmentnomean ( \diffusionscaledRWnomean (t_i)) $ 
converges almost surely to $\diffusionlimitenvironment(\diffusionlimitRW(t_i))$,  
for all $i\in\{1,\ldots, m\}$. On the intersection of these events of probability one, 
the joint convergence for all $i \in \{1,\ldots,m\}$ holds. This implies the desired 
distributional convergence.
\end{proof}

\subsection{Proof of Theorem \ref{th-2}: Limit theorems for 
\texorpdfstring{$\beta\in(1,2)$}{beta between one and two}} 

\label{section-functional}

Although Theorem \ref{th-2} was stated after Theorem \ref{th-conv-j2}, we
give the proof of the former first, because it is simpler and somehow
preliminary to the proof of the latter.
As a matter of fact, we only prove assertion {\itshape 1}. Assertion {\itshape 2}
is carried out similarly with no additional effort.

\begin{lem} \label{lem-marco}
The composition map $h: \cadlag_0 \times \cadlag^+ \into \cadlag^+$ (see 
Section \ref{subs-2-2} for the definitions of $\cadlag_0$ and $\cadlag^+$) 
defined by
\begin{equation} \label{composition-map}
  h(w,s) := w \circ s
\end  {equation}
is measurable and it is $J_1$-continuous on 
$(\mathcal C \cap \cadlag_0) \times \cadlag^+$, where $\mathcal C$ is the 
space of continuous functions on $\R$. More precisely, the continuity is 
intended w.r.t.\ the topology $J_1 \otimes J_1$ on the domain of $h$ and $J_1$ on 
its target space.
\end{lem}

\begin{proof}[Proof of Lemma \ref{lem-marco}]
As the measurability of $h$ is easy, we concentrate on the continuity statement.
Assume that, as $n \to \infty$, $(w_n, s_n) \to (w,s)$ in 
$(\cadlag_0 \times \cadlag^+, J_1 \otimes J_1)$, with 
$w \in \mathcal C\cap \cadlag_0$ and $s \in \cadlag^+$. This means that, for all 
$M,T>0$,
\begin{align}
\label{marco17}
w_n \to w \qquad &\text{in } (\cadlag([-M,M]),J_1), \\
\label{marco8}
s_n \to s \qquad &\text{in } (\cadlag([0,T]),J_1).
\end{align}
In particular, if we fix $T>0$, there exists a sequence $(\lambda_n)_{n\in\N}$ of 
homeomorphisms of $[0,T]$ such that
\begin{align}
\label{marco7}
\sup_{t\in[0,T]} | \lambda_n(t) - t| &\to 0, \\
\label{marco9}
\sup_{t\in[0,T]} | s_n \circ \lambda_n(t) - s(t)| &\to 0.
\end{align}
We have
\begin{align} \label{tau-cont}
& \sup_{t \in [0,T]} \vert w_n \circ s_n \circ \lambda_n (t) - w \circ s (t) \vert \nnl
&\qquad \leq \sup_{t \in [0,T]} \vert w_n \circ s_n \circ \lambda_n (t) - 
w \circ s_n \circ \lambda_n (t) \vert
+ \sup_{t \in [0,T]} \vert w \circ s_n \circ \lambda_n (t) - w \circ s (t) \vert \nnl
&\qquad \leq \sup_{u \in [-M,M]} 
\vert w_n (u) - w(u) \vert+ \sup_{t \in [0,T]} \vert w \circ s_n \circ \lambda_n (t) - 
w \circ s(t) \vert,
\end{align}
where $M=M(T) := \sup_n \, \sup_{v \in [0,T]} \vert s_n(v)\vert$. This quantity is finite 
because $s \in \cadlag^+$, and thus it is bounded on $[0,T]$, and from \eqref{marco8}. 

Now, the first of the last two terms of \eqref{tau-cont} vanishes by
\eqref{marco17} and Remark \ref{rmk-local-unif-conv}. The second term 
vanishes by \eqref{marco9} and the 
uniform continuity of $w$ on $[-M,M]$. Finally, \eqref{marco7}, \eqref{tau-cont}
and the arbitrariness of $T$ show that $ h(w_n,s_n) = w_n \circ s_n \to w \circ s = 
h(w,s)$, which is what we sought to prove.
\end{proof}

\begin{proof}[Proof of assertion {\itshape 1} of Theorem \ref{th-2}]
As defined in \eqref{Y-omegabar-Shat}, $\diffusionscaledflightnomean \defeq 
\fluidscaledenvironment \circ \diffusionscaledRWnomean$. Denoting by $h$ the
composition map as in the previous lemma, we set out to prove that
\begin{equation} \label{marco10}
\diffusionscaledflightnomean = h (\fluidscaledenvironment,
\diffusionscaledRWnomean) \dconv h ( \avgenvironment \tinyspace \id,  
\diffusionlimitRW ) =  \avgenvironment \diffusionlimitRW \qquad \text{in }
(\cadlag^+,J_1),
\end{equation}
as $n \to \infty$. We do so by applying the following extension of the 
Continuous Mapping Theorem, see e.g.\ \cite[Theorem 5.1]{Billingsley}:  
if $h: \mathcal{S} \into \mathcal{S}'$ is a measurable function between two 
metric spaces, which are also regarded as measure spaces w.r.t.\ the 
respective Borel $\sigma$-algebras, $(X_n)_{n\in\N}$ and $X$ are 
$\mathcal{S}$-valued random variables with $X_n \dconv X$, and  
$\mathbb P (X\in\mathrm{Disc}(h))=0$, where $\mathrm{Disc}(h) \subset 
\mathcal{S}$ denotes the set of discontinuities of $h$, then 
$h(X_n) \dconv h(X)$.

Note that the topological spaces $\cadlag_0$ and 
$\cadlag^+$, and thus $\cadlag_0 \times \cadlag^+$, are metrizable;
see the comment after \eqref{marco15}. From the independence of 
$\fluidscaledenvironment$ and  $\diffusionscaledRWnomean$, and by 
\eqref{convergence-bar-omega}, \eqref{convergence-hat-S}, and the 
definition of product topology,
\begin{equation} \label{conv1}
( \fluidscaledenvironment ,  \diffusionscaledRWnomean  ) \dconv 
( \avgenvironment\tinyspace \id, \diffusionlimitRW ) \qquad 
\text{in }(\cadlag_0 \times \cadlag^+,J_1 \otimes J_1).
\end{equation}
To apply the theorem and obtain \eqref{marco10} 
it remains to prove that the probability that
$( \avgenvironment\tinyspace \id, \diffusionlimitRW )$ hits a discontinuity
of $h$ is zero. But $( \avgenvironment \tinyspace \id, \diffusionlimitRW ) \in
(\mathcal C \cap \cadlag_0) \times \cadlag^+$, where $h$ is 
continuous by Lemma \ref{lem-marco}.
\end{proof}


\subsection{Proof of Theorem \ref{th-conv-j2}: Limit theorems for 
\texorpdfstring{$\beta\in(0,1)$}{beta less than one}}

\label{subs3-3}

In comparison with the proof of Theorem \ref{th-2}, the main technical 
hurdle here is that in the composition $\diffusionscaledflightnomean
= \diffusionscaledenvironmentnomean \circ \fluidscaledRW$, cf.\
\eqref{Y-omegahat-Sbar},
the inner function (also referred to as \emph{random time change})
is not increasing and one cannot use \cite[Theorem 5.1]{Billingsley}. 
We shall only prove Theorem \ref{th-conv-j2} in the 
case $\avgRW>0$, as the other case is all but identical.

In view of Remark \ref{rmk-open-closed},
we need to show that, for any $T>0$, which we consider fixed throughout 
this proof, the restriction of $\diffusionscaledflightnomean$ to $[0,T)$ 
converges in $(\cadlag([0,T)), J_{2})$ to the 
restriction of $\diffusionlimitRW \circ \avgRW\tinyspace\id$ to $[0,T)$.
By a double use of the Skorokhod Representation Theorem, there exist 
two probability spaces $(\Omega_1, \mathbb P_1)$ and 
$(\Omega_2, \mathbb P_2)$, and processes
\begin{equation}\label{dimj2:processo_omega}
\begin{array}{l}
\boldsymbol{\diffusionscaledenvironmentnomean}:\Omega_1\to 
\cadlag,\\
\boldsymbol{\diffusionlimitenvironment}: \Omega_1\to \cadlag,\\
\boldsymbol{\fluidscaledRW}: \Omega_2\to \cadlag^+,\\
\boldsymbol{\tilde{S}^{(n)}}: \Omega_2\to \cadlag^+,\\
\boldsymbol{\widetilde{W}^{(\alpha)}}: \Omega_2\to \cadlag^+,
\end{array}
\end{equation}
with, respectively, the same distributions as $\diffusionscaledenvironment, 
\diffusionlimitenvironment, \fluidscaledRW, \tilde{S}^{(n)}, 
\widetilde{W}^{(\alpha)}$, such that the distributional convergences
\eqref{convergence-hat-omega}, \eqref{convergence-bar-S} and 
\eqref{convergence-tildeS} become almost sure convergences in the 
suitable spaces:

\begin{equation}
\begin{array}{lll}
\boldsymbol {\diffusionscaledenvironmentnomean} \asconv 
 \boldsymbol{\diffusionlimitenvironment} &\text{on } 
(\Omega_{1},\mathbb{P}_1)& \inDJ{1}  \\
\boldsymbol{\fluidscaledRW} \asconv \avgRW \tinyspace \id 
&\text{on } (\Omega_{2},\mathbb{P}_2)& \inDPJ{1},\\

\boldsymbol{\tilde{S}^{(n)}} \asconv \boldsymbol{\widetilde{W}^{(\alpha)}} 
& \text{on } (\Omega_{2},\mathbb{P}_2) &\inDPJ{1}.
\end{array}
\end{equation}
%
%
Since $\diffusionscaledenvironmentnomean$ and $\bar{S}^{(n)}$ are 
independent, we regard the processes 
\eqref{dimj2:processo_omega} as defined on 
$(\Omega_{1}\times\Omega_2, \mathbb{P}_{1}\times \mathbb{P}_{2})$,
so all the joint distributions of processes in boldface type are the 
same as for the corresponding processes in regular type. Also, in the 
interest of readability and confident there will be no confusion, we 
slightly abuse the notation and write the boldface processes in regular 
type.

Let us denote by $\Omega_1'$ the set of realizations
$\gamma_{1}\in\Omega_{1}$ such that
$\diffusionscaledenvironmentnomean{[\gamma_1]} \to
\diffusionlimitenvironment{[\gamma_1]}$, as $n\to\infty$, in $(\cadlag, J_1)$. 
In particular $\mathbb{P}_1(\Omega_1')=1$.
Similarly, let us denote by $\Omega_2'$ the set of realizations 
$\gamma_{2}\in\Omega_{2}$ such that $\diffusionscaledRW{[\gamma_2]} \to
\widetilde{W}^{(\alpha)}{[\gamma_2]}$ in $(\cadlag^+, J_1)$. Again
$\mathbb{P}_2(\Omega_2')=1$.
Since $\diffusionscaledRW = (n\fluidscaledRW - \avgRW 
\lfloor n\, \cdot \rfloor)/n^{1/\alpha}$ converges almost surely to a L\'evy stable 
process, whose trajectories are bounded when restricted to $[0,T)$
(though not uniformly bounded in $\gamma_2$), it is easy to see that, 
for any $\eta \in (0,1)$, there exist $C_{\eta}>0$ and $\bar n_\eta \in\Naturals$ 
such that the event
\begin{equation} \label{B_eta}
B_{\eta} \defeq \Big\{ \gamma_2\in\Omega'_2 : \sup_{t\in[0,T)}
\frac{n\vert \fluidscaledRW {[\gamma_2]} (t) - \avgRW \tinyspace t\vert}
{n^{1/\alpha}} \leq C_{\eta} \text{ for } n\geq  \bar n_\eta\Big\}
\end{equation}
has probability $\mathbb P_2(B_{\eta}) > 1-\eta$. Observe that 
$C_\eta$ and $\bar n_\eta$ depend on $T$ as well. 
Our goal for the rest of the proof will be to show that, for all realizations 
$(\gamma_1,\gamma_2)\in\Omega_{1}'\times B_{\eta}$, 
\begin{equation}
\label{goal_proof_convj2_inter}
\diffusionscaledenvironmentnomean\circ\fluidscaledRW 
\longrightarrow \diffusionlimitenvironment\circ\avgRW\tinyspace\id 
\qquad \text{in } (\cadlag([0,T)),J_2),
\end{equation}
as $n \to \infty$. This easily implies that the above convergence holds
almost surely in $(\Omega_{1}\times\Omega_{2},
\mathbb{P}_1\times\mathbb{P}_2)$. In fact, it holds on $\Omega_1'
\times \bigcup_{k\in\Naturals} B_{\eta_k}$, where $(\eta_k)_{k \in \N}$
is some vanishing sequence of numbers in $(0,1)$, and
\begin{equation}
\mathbb{P}_{2} \! \left( \bigcup_{k\in\Naturals} B_{\eta_k} \right) \ge
\limsup_{k \to \infty} \, \mathbb{P}_2(B_{\eta_k}) \ge \limsup_{k \to \infty}
\, (1-\eta_k)=1.
\end{equation}
Now, almost sure convergence implies distributional convergence (for the 
boldface processes and thus for the original processes). Finally, since 
$\diffusionlimitenvironment_\pm$ are $\beta$-stable and having assumed
that $\avgRW>0$, we observe that $\diffusionlimitenvironment \circ \avgRW
\tinyspace\id =\mu^{1/\beta} Z^{(\beta)}_+$, thus proving \eqref{eq-conv-j2}.

So we are left with proving \eqref{goal_proof_convj2_inter}. We first
give a plan of the proof, warning the reader that all the involved 
quantities depend in general on $(\gamma_1,\gamma_2) \in 
\Omega_1'\times B_{\eta}$, but we often omit this dependence. 
As per Definition \ref{def-j2}, we will need to construct bijections 
$\tau_n:[0,T) \into [0,T)$ such that, when $n \to \infty$,
\begin{align}
&\sup_{t\in[0,T)} 
| \diffusionscaledenvironmentnomean \circ \fluidscaledRW 
\circ \tau_n(t) - \diffusionlimitenvironment (\avgRW t) | \to 0, 
\label{newgoal1} \\
&\sup_{t\in[0,T)}  | \tau_n(t) - t| \to 0.
\label{newgoal2}
\end{align}

\begin{itemize}
\item[1.]
As a first step to obtain $(\tau_n)_{n\in\N}$, we construct bijections $\rho_n :
[0,T) \into [0,T)$ such that
\begin{align}
&\sup_{t\in[0,T)} |\rho_n(t) - t| \longrightarrow 0 \\
\label{marco1}
&\fluidscaledRW\circ \rho_n \in \cadlag_{\mbox{pc}}([0,T)) \cap
\cadlag_0([0,T)), \\
\label{marco2}
&\fluidscaledRW\circ \rho_n \longrightarrow \avgRW\tinyspace \id 
\qquad \text{in } (\cadlag([0,T)), J_1),
\end{align}
where $\cadlag_{\mbox{pc}}([0,T))$ denotes the set of \cadlag piecewise constant 
functions of $[0,T)$ and $\cadlag_0([0,T))$ denotes the set of \cadlag 
nondecreasing functions of $[0,T)$. See Figure \ref{fig:example_time_change} 
(upper panel) for an example of $\rho_n$ associated to a given realization 
of $\fluidscaledRW$.
\begin{figure}[t!]
\begin{center}
\begin{tabular}{c}
\includegraphics[scale=0.9]{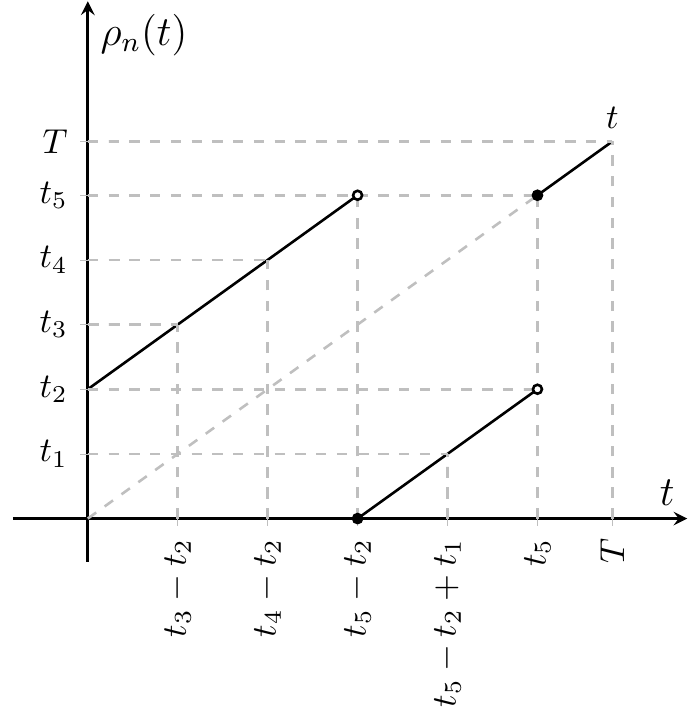}
\end{tabular}
\begin{tabular}{cc}
\includegraphics[scale=0.8]{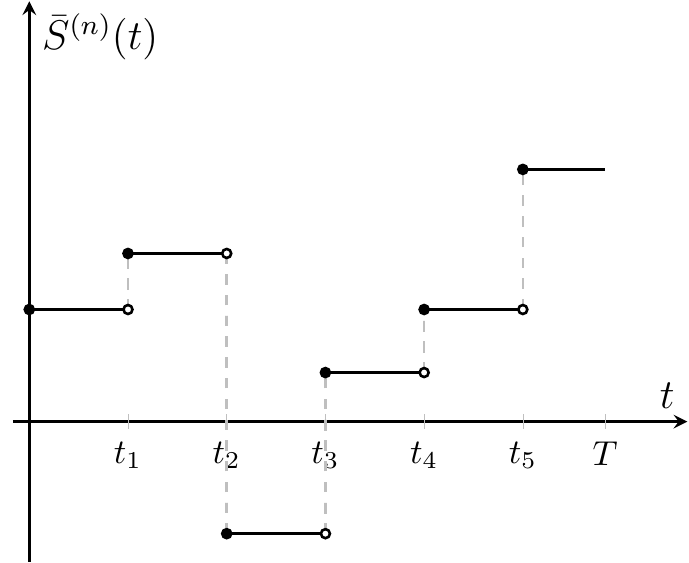} &
\includegraphics[scale=0.8]{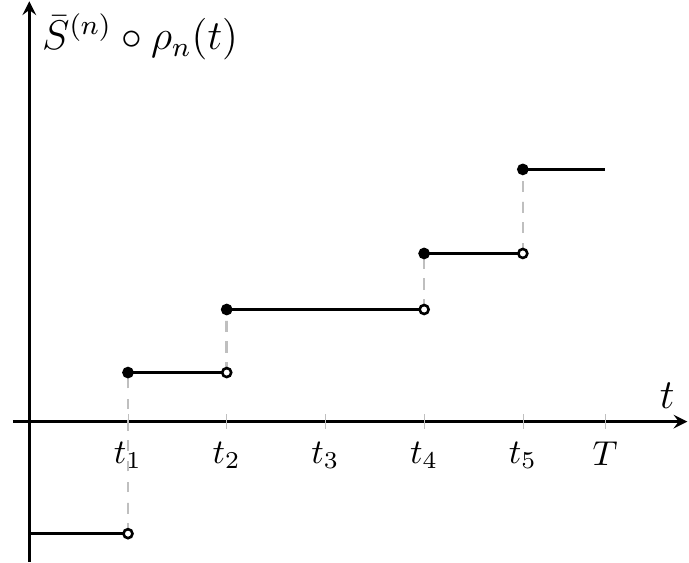}
\end{tabular}
\end{center}
\caption{Upper panel: representation of $\rho_n$ in $[0,T)$. Lower panel: sample path of $\fluidscaledRW$ in $[0,T)$ (left) and the nondecreasing composition $\fluidscaledRW \circ \rho_n$ (right).}\label{fig:example_time_change}
\end{figure}
\item[2.] Since $\diffusionscaledenvironmentnomean \longrightarrow 
\diffusionlimitenvironment$ in $(\cadlag,J_1)$ for any $\gamma_1\in\Omega'_1$, 
we can apply \cite[Theorem 3.1]{whitt1980some}, which gives sufficient 
conditions for the composition of two \cadlag functions to be
continuous in the $J_1$ topology. Using \eqref{marco1}-\eqref{marco2} 
we will get
\begin{equation}
\label{goal_proof_convj1_inter}
\diffusionscaledenvironmentnomean\circ\fluidscaledRW \circ\rho_n 
\longrightarrow \diffusionlimitenvironment\circ\avgRW\tinyspace\id  
\qquad \text{in } (\cadlag([0,T)), J_1),
\end{equation}
for any $(\gamma_1,\gamma_2)\in\Omega'_{1}\times\,B_{\eta}$.
By definition of $J_1$-convergence, there exists a sequence of 
homeomorphisms $\lambda_n$ of $[0,T)$, such that, as $n\to\infty$,
\begin{align}
\label{lambda_conv_id}
&\sup_{t\in[0,T)} | \lambda_n(t)-t | \to 0,\\
\label{eq:conv_J2:c2}%
&\sup_{t\in[0,T)} | \diffusionscaledenvironmentnomean \circ
\fluidscaledRW\circ\rho_n\circ\lambda_n(t) -
\diffusionlimitenvironment (\avgRW t) | \to 0.
\end{align}

\item[3.]  Note that \eqref{eq:conv_J2:c2} is exactly \eqref{newgoal1} 
for the bijection $\tau_n := \rho_n\circ\lambda_n$. Therefore, it will
remain to establish \eqref{newgoal2}, that is,
\begin{equation} \label{marco5}
\sup_{t\in[0,T)} | \rho_n\circ\lambda_n(t) -t | \to 0.
\end{equation}
\end{itemize}
We now fill the gaps in the steps above.

\paragraph{Construction of $\rho_n$.}
We begin by constructing the bijection $\rho_n : [0,T) \into [0,T)$.
Since $\fluidscaledRW(t) = \RW[\lfloor n t \rfloor]/n$ 
for $t\in[0,T)$ and $\gamma_2 \in B_{\eta}$, we have that 
\begin{equation} \label{eq:RW_uniformly_close_id}
\vert\RW[\lfloor u \rfloor]- \avgRW u\vert\leq C_\eta\tinyspace n^{1/\alpha},
\end{equation}
uniformly for $u\in[0,nT)$. Without loss of generality, we assume that 
$nT\in\Naturals$. If not, one can take $T'$ slightly larger than $T$, with 
$nT'\in\Naturals$, and work in $[0,T']$. 
Let $p(\cdot)$ be a permutation of $\{0,1,\dots, nT-1\}$ such that $S_{p(0)} \leq S_{p(1)}\leq\dots \leq\,S_{p(nT-1)}$. We define $\rho_n$ as follows: 
%
\begin{equation}
\rho_n(t) \defeq t - \frac{i}{n} + \frac{p(i)}{n},\qquad t\in\Big[\frac{i}{n}, \frac{i+1}{n}\Big),~i\in\{0,\ldots,nT-1\}.
\end{equation}
Clearly, $\rho_n$ is a bijection that maps $[i/n, (i+1)/n)$ affinely onto 
$[p(i)/n, (p(i)+1)/n)$. By construction of $p(\cdot)$, $\fluidscaledRW\circ\rho_n$ 
is nondecreasing. The next proposition shows that $\rho_n$ is uniformly close to the identity.
\begin{prop} \label{prop:ro_close_id}
For all $\gamma_2 \in B_{\eta}$, 
\begin{align}\label{eq:ro_close_id}
\sup_{i\in\{0,\ldots, nT-1\}}\vert i -p(i)\vert \leq\frac{2C_{\eta}}{\avgRW} \tinyspace
n^{1/\alpha} + 1.
\end{align}
\end{prop}
\begin{proof}
To better explain the proof we refer to Figure 
\ref{fig:discontinuous_bijection_close_to_identity}, which shows a sample 
path of $u\mapsto S_{\lfloor u\rfloor}$ and the corresponding upper and lower 
bounds given by \eqref{eq:RW_uniformly_close_id}.

\begin{figure}[t!]
\begin{center}
\begin{tabular}{c}
\includegraphics[scale=1.1]{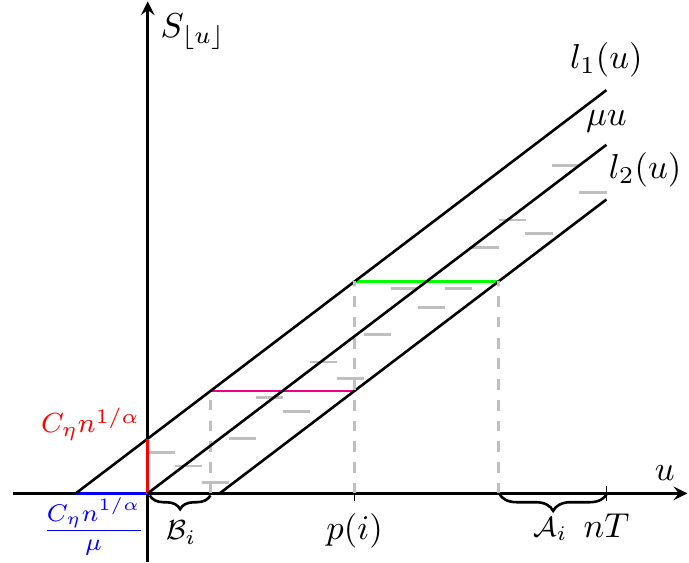}
\end{tabular}
\end{center}
\caption{A sample path of $S_{\lfloor u\rfloor}$ and the corresponding upper and 
    lower bounds given by \eqref{eq:RW_uniformly_close_id} and represented by the 
    graphs of the functions $l_{1}(u)=\mu u + C_{\eta}n^{1/\alpha}$ and 
    $l_{2}(u)=\mu u - C_{\eta}n^{1/\alpha}$. The horizontal and verticals sizes of the 
    strip are, respectively, $2 C_\eta n^{1/\alpha}/\avgRW$ and 
    $2 C_\eta n^{1/\alpha}$. The sets $\mathcal B_i, \mathcal A_i$ are defined
    in the body of the text.}
 \label{fig:discontinuous_bijection_close_to_identity}
\end{figure}


%
%
%
We establish \eqref{eq:ro_close_id} by estimating from below the cardinality of 
the sets
\begin{align}
L_i\defeq\{j \in \{0,\ldots,nT-1\} : S_{p(i)} \ge S_j\},\\
U_i\defeq\{j \in \{0,\ldots,nT-1\} : S_{p(i)} \le S_j\}.
\end{align}
Let us begin with a lower bound for $\vert L_i\vert$. To this end, consider the 
interval $\mathcal B_i := [0,b_{i,n}] \subseteq[0,nT)$, where $b_{i,n} := p(i) - 
2 C_{\eta}n^{1/\alpha}/\avgRW$. This set was introduced because it has the 
property that, for all $u \in \mathcal B_i$, $S_{\lfloor u \rfloor} \le S_{p(i)}$, see 
Figure \ref{fig:discontinuous_bijection_close_to_identity}. Observe that, for 
small values of $p(i)$, $\mathcal B_i$ might be empty. These considerations
show that
\begin{align}\label{eq:nr_jumps_less_p_i}
\vert L_i\vert \geq \vert \mathcal B_i \cap L_i\vert = 
\vert \mathcal B_i \cap \Z \vert = \max \{ \lfloor b_{i,n}\rfloor + 1, 0 \}.
\end{align}
On the other hand, since $S_{p(i)}$ is the $(i+1)$-th smallest value of the set 
$\{S_j\}_{j=0}^{nT-1}$, we know that  $\vert L_i\vert = i+1$. From the above
inequality, then,
\begin{align}
i \geq \lfloor b_{i,n}\rfloor
=
\Big\lfloor p(i) - 2 \frac{C_{\eta}n^{1/\alpha}}{\avgRW} \Big\rfloor
\geq
p(i) - 2 \frac{C_{\eta}n^{1/\alpha}}{\avgRW} -1.
\end{align}
%

We proceed analogously to produce a lower bound for $\vert U_i\vert$. Set 
$\mathcal A_i := [a_{i,n}, nT)$, where $a_{i,n} = p(i) + 2 C_{\eta} n^{1/\alpha}/\avgRW$.
Figure \ref{fig:discontinuous_bijection_close_to_identity} shows that 
$S_{\lfloor t \rfloor} \ge S_{p(i)}$ for all $t \in \mathcal A_i$, whence 
$\vert \mathcal A_i \cap  U_i\vert = \lfloor nT - a_{i,n}\rfloor +1$. On the other hand,
$\vert U_i \vert = nT - i +1$. Since $\vert U_i\vert \geq \vert \mathcal A_i \cap  U_i\vert$ 
we get
\begin{align}
nT - i \geq \lfloor nT - a_{i,n}\rfloor \geq nT - a_{i,n} -1,
\end{align}
and so
\begin{align}
i \leq a_{i,n} +1 = p(i) + 2 \frac{C_{\eta} n^{1/\alpha}}{\avgRW} +1,
\end{align}
concluding the proof.
%
\end{proof}

\paragraph{$J_2$-convergence on $\Omega'_{1}\times\,B_{\eta}$.}
To derive \eqref{goal_proof_convj1_inter}
we first notice that, for any $\gamma_2 \in B_{\eta}$,
\begin{equation} \label{marco3}
\fluidscaledRW\circ\rho_n \to \avgRW\tinyspace \id
\qquad \text{in } (\cadlag([0,T)), J_1),
\end{equation}
as $n\to\infty$. This is in fact a consequence of the following uniform 
convergence:
\begin{equation}
\begin{split}
&\sup_{t\in[0,T)}\vert\fluidscaledRW\circ\rho_n(t)-\avgRW\tinyspace t\vert \\
&\qquad = \sup_{i\in\{0,\ldots, nT-1\}}\sup_{t\in[i/n,(i+1)/n)} 
\Big\vert\fluidscaledRW\Big(t-\frac{i-p(i)}{n}\Big)-\avgRW\tinyspace t\Big\vert \\
&\qquad \leq \sup_{i\in\{0,\ldots, nT-1\}}\sup_{t\in[i/n,(i+1)/n)} \Big\vert
\fluidscaledRW\Big(t-\frac{i-p(i)}{n}\Big)-\avgRW\Big(t-\frac{i-p(i)}{n}\Big)\Big\vert \\
&\qquad\quad+ \sup_{i\in\{0,\ldots, nT-1\}}\Big\vert \frac{i-p(i)}{n}\Big\vert \\
&\qquad = C_{\eta}\frac{n^{1/\alpha}}{n} + \frac{2C_{\eta}}{\avgRW}\frac{n^{1/\alpha}}{n}+ \frac{1}{n},
\end{split}
\end{equation}
which vanishes for $n\to\infty$. Now the plan is to once again apply 
\cite[Theorem 3.1]{whitt1980some} to show that \eqref{marco3} and the
limit $\diffusionscaledenvironmentnomean \to
\diffusionlimitenvironment$ in $(\cadlag,J_1)$, which holds because 
$\gamma_1\in\Omega_1'$, imply \eqref{goal_proof_convj1_inter}. 

There is a problem, however. The theorem cannot be applied \emph{tout
court} because neither $\diffusionscaledenvironmentnomean$
nor $\diffusionlimitenvironment$ are \cadlag functions, cf\
\eqref{hat-omega}-\eqref{Z}. On the other hand, we can use the 
considerations of Remark \ref{rmk-cadlag-ified}
to show that $\diffusionscaledenvironmentnomean_\mathrm{cadlag} \to
\diffusionlimitenvironment_\mathrm{cadlag}$ in $(\cadlag,J_1)$ and
thus, by \cite[Theorem 3.1]{whitt1980some},
\begin{equation}
\diffusionscaledenvironmentnomean_\mathrm{cadlag} \circ
\fluidscaledRW \circ\rho_n \longrightarrow 
\diffusionlimitenvironment_\mathrm{cadlag} \circ\avgRW\tinyspace\id 
\qquad \text{in } (\cadlag([0,T)), J_1).
\end{equation}
But the restrictions of $\diffusionlimitenvironment_\mathrm{cadlag}$ and
$\diffusionlimitenvironment$ to $[0,\avgRW T)$ coincide, so will obtain 
\eqref{goal_proof_convj1_inter} when we prove that 
$\diffusionscaledenvironmentnomean_\mathrm{cadlag} \circ
\fluidscaledRW \circ \rho_n$ is asymptotic to 
$\diffusionscaledenvironmentnomean \circ\fluidscaledRW \circ \rho_n$ in
$(\cadlag([0,T)), J_1)$, as $n \to \infty$. We will show more, namely that,
for some $C>0$,
\begin{equation} \label{marco4}
\sup_{t \in \R^+} \left| \diffusionscaledenvironmentnomean_\mathrm{cadlag} 
\circ \fluidscaledRW \circ\rho_n(t) - \diffusionscaledenvironmentnomean
\circ \fluidscaledRW \circ\rho_n(t) \right| \le \frac C {n^{1/\beta}} .
\end{equation}
In fact, with the help of \eqref{omega-n} and \eqref{hat-omega} observe that
\begin{equation}
\diffusionscaledenvironmentnomean_\mathrm{cadlag} (s) -
\diffusionscaledenvironmentnomean(s) = \left\{
\begin{array}{ll}
\ds \frac{\zeta_j} {n^{1/\beta}} & \mbox{if } s = \ds \frac j n, \, j \in \Z^-  \\[10pt]
0 & \mbox{otherwise}
\end{array}
\right.\,,
\end{equation}
where the numbers $\zeta_j$, for $j \in\Z^-$, are fixed, as the realization 
$\gamma_1 \in \Omega_1'$ of the medium is fixed. Now, the realization
$\gamma_2 \in B_\eta$ of the underlying random walk is also fixed. 
Since the drift $\mu$ is positive, $S_n < 0$ occurs only for a finite number 
of times $n$. The values of these excursions below zero and 
their times are contained in this chain of inequalities $S_{p(0)} \le 
S_{p(1)} \le \ldots \le S_{p(m-1)} < 0 \le S_{p(m)}$, for some $m\in \N$.  
Thus
\begin{equation}
\sup_{t \in [0, p(m)/n)} \left|
\diffusionscaledenvironmentnomean_\mathrm{cadlag} 
\circ \fluidscaledRW \circ\rho_n(t) - \diffusionscaledenvironmentnomean
\circ \fluidscaledRW \circ\rho_n(t) \right| \le 
\frac{\max \{ \zeta_j \}_{j=S_{p(0)} }^{-1} } {n^{1/\beta}}.
\end{equation}
Since the expression on the above l.h.s.\ is identically 0 for $t \ge p(m)/n$,
we have proved \eqref{marco4}, thus \eqref{goal_proof_convj1_inter}, thus
\eqref{lambda_conv_id}-\eqref{eq:conv_J2:c2}.

We are left to prove that $\tau_n = \rho_n\circ\lambda_n$ satisfies 
\eqref{marco5}. We do so with
the help of Proposition \ref{prop:ro_close_id}:
\begin{equation}
\begin{split}
\sup_{t\in[0,T)} \vert\rho_n\circ\lambda_n(t) -t\vert &= \sup_{s\in[0,T)} 
\vert\rho_n(s) -\lambda_n^{-1}(s)\vert \\
&= \sup_{i\in\{0,\ldots,nT-1\}}\sup_{s\in[i/n,(i+1)/n)} 
\Big\vert s - \frac{i-p(i)}{n}- \lambda_n^{-1}(s)\Big\vert \\
&\leq \sup_{s\in[0,T)} \vert s - \lambda_n^{-1}(s)\vert + 
\sup_{i\in\{0,\ldots,nT-1\}}\Big\vert\frac{p(i)-i}{n}\Big\vert,
\end{split}
\end{equation}
which converges to 0 as $n\to\infty$ by \eqref{lambda_conv_id} and 
\eqref{eq:ro_close_id}. This finally shows that 
$\diffusionscaledenvironmentnomean \circ \fluidscaledRW \to 
\diffusionlimitenvironment \circ \avgRW \tinyspace \id$ in $(\cadlag([0,T)), J_2)$ for 
all $(\gamma_1,\gamma_2)\in \Omega'_1\times B_{\eta}$, concluding the proof of
Theorem \ref{th-conv-j2}.



\subsection{Proof of Theorem \ref{th-3}: Limit theorems for the fluctuations} 
\label{subs-fluct}

Once again, we only prove the theorem in the case $\mu>0$. Using definitions 
\eqref{Omegatilde} and \eqref{Stilde} we write
\begin{equation} \label{decomposY}
n(\fluidscaledflight - \avgRW\avgenvironment \tinyspace \id)
= n^{1/\beta}\diffusionscaledenvironment \circ \fluidscaledRW +
n^{1/\alpha} \avgenvironment\diffusionscaledRW + \avgenvironment \avgRW
(\lfloor n \tinyspace \id \rfloor - n \tinyspace \id).
\end{equation}
The asymptotic behavior of \eqref{decomposY}
depends crucially on the ratio $\alpha/\beta$, hence we distinguish three cases. 
Observe that since $(\lfloor n\tinyspace \id \rfloor -n\tinyspace \id)$ is bounded, 
the last term of the above r.h.s.\ vanishes in the limit, whether we divide it by 
$n^{1/\alpha}$ or by $n^{1/\beta}$. 

\paragraph{Case $\alpha<\beta$.}

Substituting $\diffusionscaledenvironmentnomean$ with 
$\diffusionscaledenvironment$ and 
$\diffusionlimitenvironment$ with $\tilde{Z}^{(\beta)}$, cf.\
\eqref{convergence-tilde-omega}, in the proof of Theorem \ref{th-conv-j2},
one shows that, for $\beta \in (1,2)$ and $n\to\infty$,
\begin{equation}\label{composizioneOmegaS}
\diffusionscaledenvironment \circ \fluidscaledRW
\dconv
\avgRW^{1/\beta}\tilde{Z}_{+}^{(\beta)}
\inDPJ{2}.
\end{equation}
Since in this case $1/\alpha>1/\beta$, we obtain
\begin{equation} \label{marco11}
\frac{\diffusionscaledenvironment \circ \fluidscaledRW}
{n^{1/\alpha-1/\beta}} \dconv 0 \inDPJ{1}, 
\end{equation}
where we have observed that we can pass from $J_2$-convergence to 
$J_1$-convergence because both convergences reduce to uniform
convergence when the limit function is continuous, cf.\ Remark
\ref{rmk-local-unif-conv}. Finally, by 
\eqref{decomposY}, \eqref{marco11} and \eqref{convergence-tildeS},
\begin{equation}
\frac{n(\bar{Y}^{(n)}-\avgRW\avgenvironment\tinyspace\id)}{n^{1/\alpha}} 
\dconv \nu\widetilde{W}^{(\alpha)}\qquad \text{in }(\cadlag^+,J_1),
\end{equation}
which amounts to \eqref{eq:fCLT_alpha<beta}, as desired.

\paragraph{Case $\alpha>\beta$.}
Since in this case $1/\beta>1/\alpha$,
the leading order term in \eqref{decomposY} is the first term, whose limit has 
been identified in \eqref{composizioneOmegaS}.
We conclude that 
\begin{equation}
\frac{n(\fluidscaledflight-\avgRW\avgenvironment\tinyspace\id)}{n^{1/\beta}} 
\dconv \avgRW^{1/\beta}\,\tilde{Z}_{+}^{(\beta)} \inDPJ{2},
\end{equation}
i.e.,  \eqref{eq:fCLT_alpha>beta} holds for the case $\avgRW>0$.

\paragraph{Case $\alpha=\beta$.} Except for certain complications, the proof
of this case will follow the same ideas as that of Theorem \ref{th-conv-j2}
in Section \ref{subs3-3}. We will detail the parts that need a new argument and 
describe quickly those that are proved exactly as done earlier.


In view of \eqref{decomposY}, we rewrite our process of interest as
\begin{equation} \label{marco13}
\tilde{Y}^{(n)}=
\diffusionscaledenvironment \circ \fluidscaledRW +
\avgenvironment\diffusionscaledRW + \delta_n =
\ell \big( h \big( \diffusionscaledenvironment, \fluidscaledRW \big) , 
\avgenvironment \diffusionscaledRW \big) + \delta_n,
\end{equation}
where $h(x,y) \defeq x \circ y$ is the composition map from $\cadlag 
\times \cadlag^+$ to $\cadlag^+$ and $\ell(x,y) \defeq x + y$ is the addition map 
from $\cadlag^+ \times \cadlag^+$ to $\cadlag^+$. Also $\delta_n \defeq 
\avgenvironment \avgRW (\lfloor n \tinyspace \id \rfloor - n \tinyspace \id) / 
n^{1/\alpha}$ is a negligible term, as $n \to \infty$, in any relevant distance.
%
%
As was done in Section \ref{subs3-3}, we use the Skorokhod Representation 
Theorem twice to obtain two probability spaces $(\Omega_1, \mathbb P_1)$
and $(\Omega_2, \mathbb P_2)$, and processes
\begin{equation}\label{processi_bold_5}
\begin{array}{l}
\boldsymbol{\tilde{\omega}^{(n)}}:\Omega_1\to \mathcal{D},\\
\boldsymbol{\tilde{Z}^{(\alpha)}}: \Omega_1\to \mathcal{D},\\
\boldsymbol{\bar{S}^{(n)}}: \Omega_2\to\, \mathcal{D}^+,\\
\boldsymbol{\tilde{S}^{(n)}}: \Omega_2\to\, \mathcal{D}^+,\\
\boldsymbol{\widetilde{W}^{(\alpha)}}: \Omega_2\to \mathcal{D}^+,
\end{array}
\end{equation}
with respectively the same distribution as $\tilde{\omega}^{(n)}, 
\tilde{Z}^{(\alpha)}, \bar{S}^{(n)}, \tilde{S}^{(n)}$ and $\widetilde{W}^{(\alpha)}$,
and such that
\begin{equation}
\begin{array}{lll}
\boldsymbol {\tilde{\omega}}^{(n)} \asconv 
\boldsymbol{\tilde{Z}^{(\alpha)}} &\text{on } 
(\Omega_{1},\mathbb{P}_1) &\inDJ{1} \\
\boldsymbol{\fluidscaledRW} \asconv \avgRW \tinyspace \id 
&\text{on } (\Omega_{2},\mathbb{P}_2) & \inDPJ{1}\\
\boldsymbol{\tilde{S}^{(n)}} \asconv \boldsymbol{\widetilde{W}^{(\alpha)}} 
& \text{on } (\Omega_{2},\mathbb{P}_2) & \inDPJ{1}
\end{array}\,
\end{equation}

Once again, since the processes relative to the medium and those relative to 
the dynamics are independent, it is correct to regard all boldface processes as 
defined on $(\Omega_1 \times \Omega_2, \mathbb P_1 \times \mathbb P_2)$. 
Again we simplify the notation and use the regular typeset for all 
processes \eqref{processi_bold_5}. Let us
define 
\begin{align}
\Omega_1' & \defeq \{ \gamma_1 \in \Omega_1 \,:\, 
\diffusionscaledenvironment{[\gamma_1]} \to \tilde{Z}^{(\alpha)}[\gamma_1] 
\mbox{ in } (\cadlag, J_1) \}, \\
\Omega_2' & \defeq \{ \gamma_2 \in \Omega_2 \,:\, 
\diffusionscaledRW{[\gamma_2]} \to \widetilde{W}^{(\alpha)}[\gamma_2] 
\mbox{ in } (\cadlag^+, J_1) \}. 
\end{align}
These are full-measure sets in their respective spaces. Notice
that (essentially by the definition of $\diffusionscaledRW$)
$\fluidscaledRW{[\gamma_2]} \to \avgRW \tinyspace \id$, for all
$\gamma_2 \in \Omega_2'$. Now let us fix $T>0$. We already know that
for any $\eta \in (0,1)$, there exist $C_{\eta}>0$ and $\bar n_\eta \in\Naturals$ 
(both numbers depending on $T$ as well) such that the set
\begin{equation}
B_{\eta} \defeq \Big\{ \gamma_2\in\Omega'_2 : \sup_{t\in[0,T)}
\frac{n\vert \fluidscaledRW {[\gamma_2]} (t) - \avgRW \tinyspace t\vert}
{n^{1/\alpha}} \leq C_{\eta} \text{ for } n\geq  \bar n_\eta\Big\}
\end{equation}
has measure $\mathbb P_2(B_{\eta}) > 1-\eta$. 
Now one proceeds 
as in the proof of Theorem \ref{th-conv-j2}, using 
$\diffusionscaledenvironment$ and $\tilde{Z}^{(\alpha)}$ in place of 
$\diffusionscaledenvironmentnomean$ and 
$\diffusionlimitenvironment$, respectively. The fact that now 
$\beta = \alpha \in (1,2)$ causes no breaks in the proof. One obtains that, 
for all $(\gamma_1,\gamma_2) \in \Omega'_1\times B_{\eta}$,
\begin{equation} \label{marco12}
h(\diffusionscaledenvironment, \fluidscaledRW) =
\diffusionscaledenvironment\circ\fluidscaledRW \to \tilde{Z}^{(\alpha)} 
\circ \avgRW \tinyspace \id \qquad \text{in } (\cadlag([0,T)), J_2).
\end{equation}
By construction of $B_\eta$,  $\avgenvironment \diffusionscaledRW \to 
\avgenvironment \widetilde{W}^{(\alpha)}$ in $(\cadlag([0,T)), J_2)$ for 
all $\gamma_2 \in B_\eta$. If we are able to prove that,
for all $(\gamma_1,\gamma_2) \in \Omega'_1\times B_{\eta}$, 
the addition map $\ell$ is continuous at 
$( \tilde{Z}^{(\alpha)} \circ \avgRW \tinyspace \id ,\,  \avgenvironment 
\widetilde{W}^{(\alpha)})$ in the space $\cadlag([0,T)), J_2)$, we obtain by 
\eqref{marco13} that
\begin{equation} \label{finalgoal}
\lim_{n\to\infty} \tilde{Y}^{(n)} = \lim_{n\to\infty} \ell \big( 
h \big(\diffusionscaledenvironment, \fluidscaledRW \big), 
\avgenvironment\diffusionscaledRW \big) = 
\tilde{Z}^{(\alpha)} \circ \avgRW \tinyspace \id + \avgenvironment 
\widetilde{W}^{(\alpha)},
\end{equation}
for all realizations $(\gamma_1,\gamma_2) \in \Omega'_1\times B_{\eta}$.
Since $\eta \in (0,1)$ is arbitrary, the above limit extends to an 
almost sure limit in $(\Omega_1 \times \Omega_2, \mathbb P_1 \times 
\mathbb P_2)$. Passing to distributional convergence and (freely) varying 
the choice of $T$, we finally achieve \eqref{eq:fCLT_alpha_equals_beta}
for the case $\avgRW>0$. This ends the proof of Theorem \ref{th-3}.

It remains to show that the addition map on $(\mathcal{D}([0,T)), J_2)$
is continuous under standard conditions.
For a general space of \cadlag functions, 
Whitt proves that $\ell$ is continuous, relative to the topologies $J_1$, $M_1$ or 
$M_2$, at all pairs $(x,y)$ such that $\text{Disc}(x) \cap \text{Disc}(y) = \emptyset$
(see \cite[Theorem 4.1]{whitt1980some} and \cite[Corollaries 12.7.1 \& 
12.11.5]{Whitt}, respectively). In Theorem \ref{thm-j2cont} of the Appendix
we extend this statement to the topology $J_2$. As for its applicability to our 
case, notice that we can indeed assume that, for all $(\gamma_1,\gamma_2) \in 
\Omega'_1\times B_{\eta}$,
\begin{equation} \label{cond:discont}
\text{Disc}(\widetilde{Z}^{(\alpha)} \circ \avgRW\tinyspace \id) \cap 
\text{Disc}(\avgenvironment \tilde{W}^{(\alpha)}) = \emptyset,
\end{equation}
because $\widetilde{Z}^{(\alpha)}$ and $\tilde{W}^{(\alpha)}$ are
independent $\alpha$-stable processes and one can always remove from 
$\Omega_1' \times \Omega_2'$ the null set of realizations for which 
\eqref{cond:discont} does not hold. 

%
%

\appendix

\section{Appendix: Continuity of the addition map in $J_2$}

\begin{thm} \label{thm-j2cont}
Let $I$ be a (closed, open or half-open) bounded interval. The addition 
map $\ell: \cadlag(I) \times \cadlag(I) \into \cadlag(I)$ defined by 
$\ell(x,y) \defeq x+y$ is measurable and it is $J_2$-continuous 
at all pairs $(x,y)$ such that
\begin{equation}
\mathrm{Disc}(x) \cap \mathrm{Disc}(y) = \emptyset.
\end{equation}
The latter assertion amounts to the claim that $d_{J_2, I}(x_n, x)\to 0$ 
and $d_{J_2, I}(y_n, y)\to 0$, as $n\to\infty$, imply 
$d_{J_2, I}(x_n + y_n, x+y) \to 0$.
\end{thm}

\begin{proof}
We follow the same line of arguments as in the proof of 
\cite[Theorem 4.1]{whitt1980some}. In fact, the measurability of
$\ell$ is proved exactly as in the referenced theorem. As for the 
continuity claim, we fix $I \defeq [a,b)$, which is the case needed
in Section \ref{subs-fluct}. The other three cases, $I = [a,b]$, 
$I = (a,b]$ or $I = (a,b)$, are proved exactly in the same way.
Without loss of generality, we also assume to work with 
\cadlag functions (as opposed to functions that have a \cadlag and a 
c\`agl\`ad restriction). This case happens, e.g., if $0 \le a < b$.

For a fixed $\eps>0$, we must show that there exist $\bar n \in \Z^+$ and,
for all $n \ge \bar n$, bijections $\l_n : [a,b) \into [a,b)$ such that
\begin{align}
\label{gamma-1}
\sup_{t \in [a,b)} | \l_n(t) - t | &< \eps, \\
\label{gamma-2}
\sup_{t \in [a,b)} | (x_n + y_n) \circ \l_n(t) - (x+y)(t) | &< \eps.
\end{align}
Since $x,y$ are c\`adl\`ag, by \cite[Chapter 3, Lemma 1]{Billingsley} there 
exist two finite sets of points 
\begin{align}
\mathcal{P}_x&=\{a=t_0,t_1,...,t_n,t_{n+1}=b\} \\
\mathcal{P}_y&=\{a=s_0,s_1,...,s_m,s_{m+1}=b\}
\end{align} 
such that, for all $i=1,\ldots, n+1$ and $j=1,\ldots, m+1$,
\begin{align}
\label{mod_cont1}
&\sup_{q_{1},q_{2} \in [t_{i-1}, t_i)} |x(q_{1}) - x(q_{2})| < \frac \eps 8, \\
\label{mod_cont2}
&\sup_{q_{1},q_{2} \in [s_{i-1}, s_i)}|y(q_{1}) - y(q_{2})| < \frac \eps 8.
\end{align}
From this construction we have that the discontinuity points of $x$ 
(respectively $y$) with jump size bigger than $\eps/8$ are contained in 
$\mathcal{P}_x$ (respectively $\mathcal{P}_y$). By hypothesis these
two sets of points are disjoint. Moreover, we can select the other points of 
$\mathcal{P}_x$ and $\mathcal{P}_y$ so that $\mathcal{P}_{x} \cap 
\mathcal{P}_{y} =\{a,b\}$. 
Let $4 \delta$ be the distance between the closest pair of points of 
$\mathcal{P} \defeq \mathcal{P}_x \cup \mathcal{P}_y$.
For $i=1,\ldots,n$ and $j=1,\ldots,m$, 
we construct closed intervals $\mathcal{J}_i^{(x)}$ and 
$\mathcal{J}_j^{(y)}$ such that
\begin{align}
\label{J-x}
[t_i-\delta, t_i+\delta] \subset \mathrm{int} \big( \mathcal{J}_i^{(x)} \big) 
&\subset \mathcal{J}_j^{(y)} \subset (t_i-2\delta, t_i+2\delta), \\
\label{J-y}
[s_j-\delta, s_j+\delta] \subset \mathrm{int} \big( \mathcal{J}_i^{(x)} \big) 
&\subset \mathcal{J}_j^{(y)} \subset (s_j-2\delta, s_j+2\delta).
\end{align}
This implies in particular that these intervals are pairwise disjoint.

Now let us assume that 
there exist $\bar n\in\Z^+$ and bijections $\mu_n, \nu_n: 
[a,b)\into [a,b)$ so that, for all $n \ge \bar n$, 
\begin{align}
\label{cont_addiz_biez1}
\sup_{t \in [a,b)}|\mu_n(t) - t | < \min \{\eps,\delta \}, 
\quad \sup_{t \in [a,b)} | x_n \circ \mu_n(t)- x (t) | &< \frac \eps 4,  
\quad\mu_n(\mathcal{J}_i^{(x)}) = \mathcal{J}_i^{(x)} \\
\label{cont_addiz_biez2}
\sup_{t \in [a,b)}|\nu_n(t) - t | < \min \{\eps,\delta \},  
\quad \sup_{t \in [a,b)} | y_n\circ \nu_n(t) - y(t)  | &< \frac \eps 4, 
\quad \nu_n(\mathcal{J}_j^{(y)}) = \mathcal{J}_j^{(y)}
\end{align}
for all $i=1,\ldots, n$ and $j=1,\ldots, m$.
The first and second conditions in both \eqref{cont_addiz_biez1} and 
\eqref{cont_addiz_biez2} can be satisfied by the hypotheses
$x_n \to x$, $y_n \to y$ in $(\cadlag([a,b)), J_2)$. We postpone 
for a moment the proof that $\mu_n, \nu_n$ can be found to
satisfy the third conditions as well. Let us construct the  
bijection $\l_n:[a,b)\into [a,b)$ as follows:
\begin{equation} \label{del-lambda-n}
\l_n(t) := 
\begin{cases}
\mu_n(t) &\mbox{ for } t \in \mathcal{J}_i^{(x)} \mbox{ with } i=1,\ldots, n, \\
\nu_n(t) &\mbox{ for } t \in \mathcal{J}_j^{(y)} \mbox{ with } j=1,\ldots, m, \\
t & \mbox{ otherwise}.
\end{cases}
\end{equation}
We have the following estimates:
\begin{equation} \label{stime-gamma-Jx}
\sup_{t \in \mathcal{J}_i^{(x)}} | x_n \circ \l_n (t) - x(t) | = 
\sup_{t \in \mathcal{J}_i^{(x)}} | x_n \circ \mu_n (t) - x(t) | < \frac \eps 4 ,
\end{equation}
by \eqref{cont_addiz_biez1}, and 
\begin{equation} \label{stime-gamma-Jy}
\begin{split}
&\sup_{t \in \mathcal{J}_j^{(y)}} | x_n \circ \l_n (t) - x(t) | \\
&\qquad = \sup_{t \in \mathcal{J}_j^{(y)}} | x_n \circ \nu_n (t) - x(t) | \\
&\qquad \le \sup_{t \in \mathcal{J}_j^{(y)}} 
| x_n \circ \nu_n (t) - x \circ \mu_n^{-1} \circ \nu_n (t) | + 
\sup_{t \in \mathcal{J}_j^{(y)}} | x \circ \mu_n^{-1} \circ \nu_n (t) - x(t) | \\
&\qquad \le \sup_{u \in [a,b)} | x_n \circ \mu_n (u) - x(u) | + 
\sup_{t \in \mathcal{J}_j^{(y)}} | x (\mu_n^{-1} \circ \nu_n (t)) - x(t) | \\
&\qquad < \frac \eps 4 + \frac \eps 8 = \frac{3\eps} 8 .
\end{split}
\end{equation}
In the first term of the final estimate of \eqref{stime-gamma-Jy} we have 
renamed $u \defeq \mu_n^{-1} \circ \nu_n (t)$ and used \eqref{cont_addiz_biez1}.
For the second term we have observed that, by 
\eqref{cont_addiz_biez1}-\eqref{cont_addiz_biez2}, the bijection
$\mu_n^{-1} \circ \nu_n$ is closer to the identity than $2\delta$. Since
$t \in \mathcal{J}_j^{(y)}$ this implies, by \eqref{J-y}, that $t$ and 
$\mu_n^{-1} \circ \nu_n (t)$ belong to the same interval $[t_{i-1}, t_i)$, for some
$i$. Thus we have used \eqref{mod_cont1}. Lastly, if we denote
$\mathcal{J} \defeq \big( \bigsqcup_{i=1}^n \mathcal{J}_i^{(x)} \big) \sqcup 
\big(\bigsqcup_{j=1}^m \mathcal{J}_j^{(y)} \big)$,
\begin{equation} \label{stime-gamma-J}
\begin{split}
&\sup_{t \in [a,b) \setminus \mathcal{J} } | x_n \circ \l_n (t) - x(t) | \\
&\qquad =\sup_{t \in [a,b) \setminus \mathcal{J} } | x_n  (t) - x(t) | \\
&\qquad \le  \sup_{t \in [a,b) \setminus  \mathcal{J} } | x_n  (t) - x \circ \mu_n^{-1}(t) | 
+ \sup_{t \in [a,b) \setminus \mathcal{J} } |  x \circ \mu_n^{-1}(t) - x(t) | \\
&\qquad \le \sup_{u \in [a,b)} | x_n \circ \mu_n (u) - x(u) | + 
\sup_{t \in [a,b) \setminus \mathcal{J} } | x (\mu_n^{-1}(t)) - x(t) | \\
&\qquad < \frac \eps 4 + \frac \eps 8  = \frac{3\eps} 8 ,
\end{split}
\end{equation}
where we have used 
the same arguments as for \eqref{stime-gamma-Jy}: observe in fact that
if $t \not\in \mathcal{J}$ then, by \eqref{J-x}-\eqref{J-y}, $t$ is at distance 
larger than
$\delta$ from $\mathcal{P} \setminus \{a,b\}$. By \eqref{cont_addiz_biez1}
$| \mu_n^{-1}(t) - t| < \delta$ and so $t$ and $\mu_n^{-1}(t)$ belong to the 
same interval $[t_{i-1}, t_i)$, for some $i$, triggering \eqref{mod_cont1}.

From \eqref{stime-gamma-Jx}-\eqref{stime-gamma-J} we have that 
$\sup_{t \in [a,b) }  | x_n \circ \l_n (t) - x(t) | < \eps/2$, and the same obviously 
holds for $y$, whence
\begin{equation} 
\begin{split}
&\sup_{t \in [a,b) }  | (x_n + y_n)\circ\l_n(t) - (x+y)(t) | \\
&\qquad \le \sup_{t \in [a,b) }  | x_n\circ \l_n(t) - x(t)  | + 
\sup_{t \in [a,b) }  | y_n\circ \l_n(t) - y(t)  | < \eps ,
\end{split}
\end{equation}
giving \eqref{gamma-2}. The inequality \eqref{gamma-1} follows from 
definition \eqref{del-lambda-n} and 
\eqref{cont_addiz_biez1}-\eqref{cont_addiz_biez2}
and so Theorem \ref{thm-j2cont} is proved.

\medskip

It remains to show that the bijections $\mu_n, \nu_n$ can be chosen
to satisfy the rightmost conditions of 
\eqref{cont_addiz_biez1}-\eqref{cont_addiz_biez2},
for a suitable choice of the intervals $\{ \mathcal{J}_i^{(x)} \}_{i=1}^n$,  
$\{ \mathcal{J}_j^{(y)} \}_{j=1}^m$.  We proceed by explicitly constructing
$\mu_n$, as the construction of $\nu_n$ is completely analogous.

The hypothesis $d_{J_2, [a,b)}(x_n, x)\to 0$ amounts to the existence
of bijections $\rho_n : [a,b) \into [a,b)$ such that, for $n \ge \bar n$,
\begin{align}
\label{ell-1}
\sup_{t \in [a,b)} |\rho_n(t) - t | &< \frac12 \min \{ \eps, \delta \} , \\
\label{ell-2}
\sup_{t \in [a,b)} | x_n \circ \rho_n(t) - x(t)  | &< \frac \eps 8 .
\end{align}
For $i = 1, \ldots, n$, set $[a_i,b_i] := [t_i - \delta, t_i + \delta]$ and
\begin{equation} \label{marco19}
\mathcal{J}_i^{(x)} \defeq [a_i'',b_i''] \defeq 
\left[ \inf \! \big( \{ a_i \} \cup \rho_n( [a_i, b_i] ) \big) - \frac{\delta}{2} \,,\,
\sup \! \big( \{ b_i \} \cup \rho_n( [a_i, b_i]) \big) + \frac{\delta}{2} \right].
\end{equation}
By construction $|a_i'' - t_i| < 2 \delta$ and $|b_i'' - t_i| < 2 \delta$. Hence the 
intervals $\mathcal{J}_i^{(x)}$ satisfy \eqref{J-x}.

The bijection $\mu_n : [a,b) \into [a,b)$ is defined with the following structure:
\begin{equation} \label{def-mu-n}
\mu_n(t) := 
\begin{cases}
\mu_n^{(i)}(t) & \mbox{ for } t \in \mathcal{J}_i^{(x)} \mbox{ with } 
i=1,\ldots, n, \\
t & \mbox{ otherwise},
\end{cases}
\end{equation}
where $\mu_n^{(i)} : \mathcal{J}_i^{(x)} \into \mathcal{J}_i^{(x)}$ are 
bijections that we construct in several steps as follows. 

First, on $[a_i,b_i] \subset \mathcal{J}_i^{(x)}$, we define 
$\mu_n^{(i)} |_{ [a_i,b_i] } := \rho_n |_{ [a_i,b_i] }$. 
In light of \eqref{marco19} and applying
\eqref{ell-1} to $t \in [a_i, b_i]$, we see that
\begin{equation}
\mu_n^{(i)} ( [a_i,b_i] ) = \rho_n( [a_i, b_i] ) 
\subset (a_i'' , b_i'') \subset \mathcal{J}_i^{(x)}.
\end{equation}
Denote $\mathcal A_i := [a_i'',t_i) \setminus \mu_n^{(i)} ( [a_i,b_i] )$ and 
$\mathcal B_i := [t_i,b_i''] \setminus \mu_n^{(i)} ( [a_i,b_i] )$. These are, 
respectively, the lower and upper parts of $[a_i'',b_i''] = \mathcal{J}_i^{(x)}$ 
that have not yet been assigned as image points of $\mu_n^{(i)}$ (which is
only partially defined at this stage).
Using the fact that the inequality \eqref{ell-1} is strict, we can find 
$\eta$ with $0 < \eta < \min \{ \eps, \delta \} /2$ such that
\begin{align}
\label{marco22}
\inf \mu_n^{(i)} ( [a_i,b_i] ) >  a_i - \eta \,, &\qquad 
\sup \mathcal A_i < a_i + \eta \,, \\
\label{marco23}
\sup \mu_n^{(i)} ( [a_i,b_i] ) < b_i + \eta \,, &\qquad 
\inf \mathcal B_i > b_i - \eta .
\end{align}
Set $a_i' \defeq a_i-\eta$ and $b_i' \defeq b_i+\eta$. We have $a_i'' < 
a_i' < a_i < b_i < b_i' <b_i''$. The inequalities \eqref{marco22} show 
that the yet-to-be-assigned image set $\mathcal A_i$ can we written as
\begin{equation}
\mathcal A_i = [a_i'',a_i'] \sqcup ( \mathcal A_i \cap (a_i', a_i + \eta)),
\end{equation}
where $\mathcal A_i \cap (a_i', a_i + \eta)$ has the cardinality of the
continuum because, by the first inequality of \eqref{marco22}, there
exists $\sigma>0$ such that $(a_i', a_i' + \sigma) \subset
\mathcal A_i \cap (a_i', a_i + \eta)$. By reasons of cardinality, then, 
there exists a bijection $\phi_i^- : (a_i', a_i) \into \mathcal A_i 
\cap (a_i', a_i + \eta)$. By construction, since $a_i' = a_i - \eta$,
\begin{equation} \label{marco24}
\sup_{ t \in (a_i', a_i) } | \phi_i^-(t) - t | \le 2\eta < \min \{ \eps, \delta \}.
\end{equation}
We define $\mu_n^{(i)} |_{ (a_i',a_i) } \defeq \phi_i^-$ and
$\mu_n^{(i)} |_{ [a_i'',a_i'] } \defeq \id$. Analogously, the inequalities 
\eqref{marco23} give
\begin{equation}
\mathcal B_i = ( \mathcal B_i \cap (b_i - \eta, b_i')) \sqcup [b_i', b_i'']
\end{equation}
and there exists a bijection $\phi_i^+ : (b_i, b_i') \into 
\mathcal B_i \cap (b_i - \eta, b_i')$ for which the analogue of estimate
\eqref{marco24} holds. Finally, we define  $\mu_n^{(i)} |_{ (b_i,b_i') } \defeq 
\phi_i^+$ and $\mu_n^{(i)} |_{ [b_i',b_i''] } \defeq \id$. This completes the 
definition of $\mu_n^{(i)}$ as a bijection of $\mathcal{J}_i^{(x)}$.

By \eqref{ell-1}, \eqref{marco24} and its analogue for $\phi_i^+$, we see
that
\begin{equation} \label{close-id2}
\sup_{t \in \mathcal{J}_i^{(x)} } | \mu_n^{(i)} (t) - t | < \min \{ \eps, \delta \} .
\end{equation}
Also, by the definition of $\mu_n^{(i)} |_{ [a_i,b_i] }$ and \eqref{ell-2},
\begin{equation} \label{marco20}
\sup_{t \in [a_i,b_i]} | x_n \circ \mu_n^{(i)} (t) - x(t) | < \frac \eps 8 .
\end{equation}
Furthermore,
\begin{align}
&\sup_{t \in \mathcal{J}_i^{(x)} \setminus [a_i,b_i]} 
| x_n \circ \mu_n^{(i)} (t) - x(t) | \nonumber \\
&\qquad \le \sup_{t \in \mathcal{J}_i^{(x)} \setminus [a_i,b_i]} \!
| x_n \circ \mu_n^{(i)} (t) - x \circ \rho_n^{-1} \circ \mu_n^{(i)} (t) | 
+ \! \sup_{t \in \mathcal{J}_i^{(x)} \setminus [a_i,b_i]} \!
| x \circ \rho_n^{-1} \circ \mu_n^{(i)} (t) - x(t) | \nonumber \\[2pt]
\label{j-i}
&\qquad \le \frac \eps 8 + \frac \eps 8 = \frac \eps 4.
\end{align}
The above estimates are derived in a way similar to that used in  
\eqref{stime-gamma-Jy}: for the first term we use \eqref{ell-2} after the change
of variable $u \defeq \rho_n^{-1} \circ \mu_n^{(i)} (t)$; for the second term we
use \eqref{mod_cont1} and the fact that $t$ and $\rho_n^{-1} \circ \mu_n^{(i)} (t)$
belong to the same interval $[t_{k-1}, t_k)$, for some $k$ (this is because,
due to \eqref{ell-1} and \eqref{close-id2}, $t \in \mathcal{J}_i^{(x)}$ and 
$| \rho_n^{-1} \circ \mu_n^{(i)} (t) - t | < 3\delta/2$). Moreover, 
denoting $\mathcal{J}^{(x)} \defeq \bigsqcup_{i=1}^n \mathcal{J}_i^{(x)}$,
it is now easy to use \eqref{ell-2}, \eqref{J-x}  and \eqref{mod_cont1} to 
estimate
\begin{equation} \label{marco21}
\begin{split}
&\sup_{t \in [a,b) \setminus \mathcal{J}^{(x)} } | x_n \circ \mu_n (t) - x(t) | \\
&\qquad = \sup_{t \in [a,b) \setminus \mathcal{J}^{(x)} } | x_n  (t) - x(t) | \\
&\qquad \le \sup_{t \in [a,b) \setminus  \mathcal{J}^{(x)} } 
| x_n  (t) - x \circ \rho_n^{-1}(t) | \: + \sup_{t \in [a,b) \setminus 
\mathcal{J}^{(x)}  } |  x \circ \rho_n^{-1}(t) - x(t) | \\[2pt]
&\qquad < \frac \eps 8 + \frac \eps 8  = \frac \eps 4 .
\end{split}
\end{equation}

Finally, the definition \eqref{def-mu-n} of $\mu_n$ and the inequalities 
\eqref{close-id2}-\eqref{marco21} yield \eqref{cont_addiz_biez1} and
conclude the proof of Theorem \ref{thm-j2cont}.
\end{proof}

\begin{rem} \label{rmk-j32}
One can define a new topology of the Skorokhod type in the same way
as Definitions \ref{def-j1-i} and \ref{def-j2} but taking the infimum in 
\eqref{d-j1-i} over all \emph{piecewise increasing and continuous} (PIC)
bijections $\lambda: I \into I$. A PIC bijection $\lambda: I \into I$ is one
such that $I$ can be partitioned into a finite number of intervals, on each of
which $\lambda$ is increasing and continuous. Observe that in this case 
$\lambda^{-1}$ is also a PIC bijection. For want of a better name, let us call 
this topology $J_{3/2}$. Evidently, $J_{3/2}$ is weaker than $J_1$ and stronger 
than $J_2$. It is not hard to see that Theorem \ref{thm-j2cont} can be proved 
as well with the $J_{3/2}$-distance in place of the $J_2$-distance. Furthermore, 
in the proofs of Theorems \ref{th-conv-j2} and \ref{th-3}, every time we needed 
to construct a sequence of bijections in order to prove a $J_2$-convergence, 
we have indeed produced a sequence of PIC bijections. Therefore, all 
assertions in this paper that are stated for the topology $J_2$, see 
\eqref{eq-conv-j2}, \eqref{eq:fCLT_alpha>beta}, 
\eqref{eq:fCLT_alpha_equals_beta}, hold for the topology $J_{3/2}$ as well.
\end{rem}

\end{document}